\documentclass[12pt]{amsart}
\usepackage[utf8]{inputenc}
\usepackage[lite]{amsrefs}
\usepackage{amsfonts}
\usepackage{mathrsfs}
\usepackage{amscd}

\usepackage{physics}
\usepackage{amsmath}
\usepackage{mathtools}
\usepackage{amssymb}
\usepackage{multicol}
\usepackage{graphicx}
\usepackage{wrapfig}
\usepackage{float}
\usepackage[english]{babel}
\usepackage[utf8]{inputenc}
\usepackage{fancyhdr}
\usepackage{amsthm}
\usepackage[all]{xy}
\usepackage{tikz-cd}
\usepackage{bbold}
\usepackage{hyperref}
\usepackage{tabularx,booktabs,caption,ragged2e}
\newcolumntype{L}{>{\RaggedRight\arraybackslash}X}
\newtheorem{thm}{Theorem}[section]
\newtheorem{cor}[thm]{Corollary}
\newtheorem{lemma}[thm]{Lemma}
\newtheorem{propn}[thm]{Proposition}

\theoremstyle{definition}
\newtheorem{defn}[thm]{Definition}

\theoremstyle{remark}
\newtheorem{remark}[thm]{Remark}

\newcommand{\g}{\mathfrak{g}}

\newcommand{\R}{\mathbb{R}}
\newcommand{\C}{\mathbb{C}}

\newcommand{\SO}{\operatorname{SO}}

\newcommand{\su}{\mathfrak{su}}

\usepackage{graphicx} 

\title{Hyperkähler structures on leaves of hyper-Lie Poisson manifolds}
\author{Dadi Ni}
\address{School of Mathematics and Statistics, Henan University} 
\email{\href{mailto:nidd@henu.edu.cn}{nidd@henu.edu.cn}}
\author{Kaichuan Qi}
\address{Department of Mathematics, Penn State University} 
\email{\href{mailto:kaichuan@psu.edu}{kaichuan@psu.edu}}

\thanks{Ni is  supported by the Natural Science Foundation of Henan Province
 (No. 252300421766). Qi is partially supported by the National Science Foundation (award DMS-2302447).}

\begin{document}

\maketitle
\begin{abstract}
  
Due to their rich structures and close connections with gauge theory, hyperkähler manifolds have attracted increasing interest.
Using infinite dimensional hyperkähler reduction, Kronheimer proved that certain adjoint orbits of complexified semisimple Lie algebras admit hyperkähler structures. Later on, Xu obtained a proof for the existence of  hyperkähler structures on adjoint orbits of $\mathfrak{sl}(2,\mathbb{C})$ from the viewpoint of symplectic geometry. This paper aims to thoroughly investigate and elucidate the key differences as well as the underlying connections between the two distinct construction methods.
\end{abstract}

\begin{itemize}
	\item 
\end{itemize}

\tableofcontents

\section{Introduction}

\noindent $\diamond$ \textbf{Background}

The study of hyperkähler manifolds began in 1955 with Berger's classification of holonomy groups of Riemannian manifolds.  Specifically, a hyperkähler manifold is a Riemannian manifold $(M, g)$ with three complex structures $I$, $J$ and $K$ which satisfy  the quaternionic identity $IJK = -{\rm id}_{ TM}$ and $g$ is Kählerian with respect to $I,J$ and $K$.  However, unlike general Kähler manifolds, hyperkähler structures are  more rigid and difficult to construct. The main methods for constructing hyperkähler manifolds are twistor theory and hyperkähler reduction.

One of the main features of hyperkähler geometry is its deep connection with symplectic geometry. In particular, every hyperkähler manifold possesses an underlying holomorphic symplectic structure, given by  $(I,\omega_J+i\omega_K)$, where $\omega_J$ and $\omega_K$ are the Kähler forms associated with the complex structures J and K, respectively. Conversely,     Yau's solution to the Calabi conjecture \cite{Yau-CPAM} implies that compact holomorphic symplectic manifolds of Kähler type necessarily possess a hyperkähler structure. However, in the non-compact setting, no such general existence theorem exists, and the problem of constructing hyperkähler metrics on non-compact holomorphic symplectic manifolds has generated a lot of research since Calabi's first examples \cite{Calabi1979} on $T^*\mathbb{C}\mathbb{P}^n$.  

Notable examples of hyperkähler manifolds include (co)adjoint orbits of complexified real semisimple Lie algebras \cite{Kronheimer-JLMS, Kronheimer1990InstantonsAT,Kovalev1996NahmsEA,Xu}; cotangent bundles of hermitian symmetric spaces \cite{Biquard1997,Biquard1998}; the ALE spaces on resolutions of Kleinian singularities \cite{Kronheimer-ALE} later generalized to Nakajima quiver varieties \cite{Nakajima}; various moduli spaces in gauge theory, such as moduli spaces of instantons \cite{Maciocia}, monopoles \cite{Atiyah-Hitchin}, Higgs bundles \cite{Hitchin1987}. Perhaps the most general existence result is that of Feix \cite{Feix2001} and Kaledin \cite{Kaledin}, who independently showed that the cotangent bundle of any Kähler manifold has a 
hyperkähler structure on a neighborhood of its zero section. Later on, his result was generalized by Mayrand \cite{Mayrand} to neighborhood of complex Lagrangian submanifolds in holomorphic symplectic manifolds.

\noindent $\diamond$ \textbf{Motivation} 

Using gauge theory and infinite dimensional hyperkähler reduction, Kronheimer proved in \cite{Kronheimer-JLMS,Kronheimer1990InstantonsAT} that there exist hyperkähler structures on the nilpotent and regular adjoint orbits of any complexified semisimple Lie algebra. The key point is that the adjoint orbits are realized as the moduli spaces of solutions to Nahm's equations associated with a semisimple Lie algebra $\mathfrak{g}$. Solutions to Nahm's equations can be interpreted as instantons (anti-self-dual connections) on $\mathbb{R}^4\setminus \{0\}$\cite{Kronheimer1990InstantonsAT}; the hyperkähler structure is then a consequence of a rather general property of instanton moduli space. Later on, his result was generalized by Biquard and Kovalev to arbitrary adjoint orbits \cite{Kovalev1996NahmsEA, Biquard1996}, and has been used to understand the Kostant-Sekiguchi correspondence \cite{Vergne}.

By focusing on symplectic structures rather than the metrics, one obtains a new way to define the hyperkähler manifolds. A hyperkähler structure on a manifold $S$ is three symplectic structures $\omega_1$, $\omega_2$ and $\omega_3$ satisfying
\begin{itemize}
\item $[\omega_i^{\flat}\circ (\omega_j^{\flat})^{-1}]^2 = -1$, $\forall i\neq j$; 
\item The bundle map $g\colon TS\to T^*S$ given by $g=\omega_3(\omega_1)^{-1}\omega_2$ is a positive definite.
\end{itemize}
 
The above viewpoint is emphasized in Xu's work \cite{Xu}. By introducing the concept of hyper-Lie Poisson structure, Xu constructed hyperkähler manifolds related to the adjoint orbits of $\mathfrak{sl}(2,\mathbb{C})$.  For $\mathfrak{g}=\mathfrak{su}(2)$, 
Xu defined a special Poisson structure $\pi_1$ on an open subset $\mathfrak{g}^3_{\star}$ of $\mathfrak{g}^3$, such that ${\rm pr}_{12}, {\rm pr}_{13}\colon (\mathfrak{g}^3_{\star},\pi_1) \to (\mathfrak{g}^{\mathbb{C}},\pi_{\rm Lie})$ are Poisson maps, where $\pi_{\rm Lie}$ is the Lie-Poisson structure on $\mathfrak{g}^{\mathbb{C}}$ seen as a real Lie algebra. Applying the pushforward by a cyclic permutation $\sigma:~(a,b,c)\to (c,a,b)$, one can obtain the other two Poisson structures $\pi_2=\sigma_*\pi_1$ and $\pi_3=\sigma_*\pi_2$. 

Further, he proved that symplectic leaves of the three Poisson structures coincide, and every leaf admits three symplectic structures with compatibility conditions, inducing (pseudo-)hyperkähler structures. Finally, all leaves of $\pi_1, \pi_2, \pi_3$ can be regarded as (possibly submanifolds of) certain adjoint orbits of $\su(2)^{\mathbb{C}}\cong \mathfrak{sl}(2,\mathbb{C})$ under the projection ${\rm pr}_{23}$.

Since both constructions of Kronheimer and Xu yield hyperkähler structures on the adjoint orbits of $\mathfrak{sl}(2,\mathbb{C}),$ we ask two natural questions, which we would like to answer in this paper.
    \begin{itemize}
        \item Do some leaves of the hyper-Lie Poisson structure recover Kronheimer's construction?
        \item If so, do other leaves give rise to different hyperkähler structures?
    \end{itemize}  

\noindent $\diamond$ \textbf{The main results} 

The outline of this paper is as follows: in Section 2, we review the construction of the hyper-Lie Poisson structures on $\mathfrak{su}(2)^3_{\star}$, following \cite{Xu}. Then we give a simple description of its symplectic leaves in Section 3. For any $q,r \geq 0,$ denote by $N_{q,r}$ the symplectic leaf
\begin{equation*}
\begin{aligned}
\{(a_1,a_2,a_3) \in \su(2)^3 |& \langle a_1,a_1\rangle - \langle a_2,a_2\rangle = q, \langle a_2,a_2\rangle - \langle a_3,a_3\rangle = r,\\
&\forall i\neq j: \langle a_i,a_j \rangle = 0,  \langle a_1, [a_2, a_3]\rangle < 0 \}.
\end{aligned} 
\end{equation*}
We  demonstrate that all symplectic leaves of the hyper-Lie Poisson manifold $(\su(2)^3_{\star}; \pi_1,\pi_2,\pi_3)$  are diffeomorphic as smooth manifolds. For any leaf $L$, there exist $q,r$ and an (anti)-isomorphism between $L$ and $N_{q,r}$, as (pseudo-)hyperkähler structures.

In Section 4, we classify these leaves into three types: type-0 ($N_{0,0}$), type-1 ($N_{q,0}$ for $q>0$), and type-2 ($N_{q,r}$ for $r>0$). 
By computing their sectional curvatures, we prove the following:
\begin{thm}
The hyperkähler metric on symplectic leaves of the hyper-Lie Poisson manifold $(\su(2)^3_{\star}; \pi_1,\pi_2,\pi_3)$ of different types  cannot be  isomorphic as hyperkähler manifolds. Moreover, for any $t>0$,  the diffeomorphism
$
\rho_{t}\colon N_{q,r}\rightarrow N_{qt,rt},
(a_1,a_2,a_3)\mapsto (\sqrt{t}a_1,\sqrt{t}a_2,\sqrt{t}a_3),
$
intertwines the three complex structures, and satisfies  $\rho^*_t (g_{qt,rt})=\sqrt{t}\cdot g_{q,r},$ 
where $g_{qt,rt}$, $g_{q,r}$ are the induced metrics on $N_{qt,rt}$, $N_{q,r}$, respectively. 
\end{thm}
\noindent Note that the relation between two  type-2 leaves $N_{q,r}$ and $N_{q',r'}$, where  $(q,r)\nparallel(q',r')$, is unclear to the authors and is identified as future work.

Finally, in Section 5, we relate leaves of $\mathfrak{su}(2)^3_{\star}$ with Kronheimer's moduli space of solutions to Nahm's equation \cite{Kronheimer-JLMS, Kronheimer1990InstantonsAT}. Specifically, we establish that the type-0 leaf $N_{0,0}$ is isomorphic to the hyperkähler structure constructed by Kronheimer on the nilpotent orbit of $\mathfrak{sl}(2,\mathbb{C})$. Similarly, the extended type-1 leaves $\overline{N}_{q,0}$  are isomorphic to the hyperkähler structures constructed by Kronheimer on the regular orbits of $\mathfrak{sl}(2,\mathbb{C})$. 
Hence, in the case of $\g = \su(2),$ the hyperkähler structures on Kronheimer's moduli space are recovered by the type-0 and 1 symplectic leaves of the hyper-Lie Poisson structures. 
However, the type-2 leaf $N_{q,r}$  cannot be obtained through Kronheimer's construction, marking a significant distinction between Xu's work and Kronheimer's.

\noindent $\diamond$ \textbf{ Future work} 




One could study the family of type-2 leaves, and give a classification of them as hyperkähler manifolds. Moreover, one can investigate further Poisson-geometric properties of the hyper-Lie Poisson structure on $\mathfrak{su}(2)^3_{\star}$, e.g. its symplectic realisation.

Finally, since Kronheimer's moduli space construction works for any real semisimple Lie algebra, it would be interesting to see whether in those cases hyper-Lie Poisson structures still exist.

\section{Hyper-Lie Poisson structures}\label{Sec:hyper-Lie-Poisson}
In this section we review notations related to hyperkähler manifolds and the construction of hyper-Lie Poisson structures, which are preliminaries for the later sections.
\subsection{Hyperkähler manifolds}
Let $M$ be a smooth manifold. A tuple $(g,I,J,K)$ is a hyperkähler structure on $M$, if $g$ is a Riemannian metric and $I, J, K$ are three  complex structures on $M$ such that $g$ is Kählerian with respect to $I,J$ and $K$, and that $IJK = -{\rm id}_{ TM}$. Therefore, $M$ must be a $4n$-dimensional manifold. 

Now we recall an equivalent definition. Consider the three Kähler forms 
$\omega_I,\omega_J$ and $\omega_K$ of a hyperkähler structure $(g,I,J,K)$. One checks by definition that the Kähler forms determines $g,$ and therefore, hyperkähler structures can be studied by using the symplectic forms. 

\begin{defn}\label{Def:hypersymplectic}
A \textbf{hyperkähler structure} on a manifold $M$ is three symplectic forms $\omega_1$, $\omega_2$ and $\omega_3$ on $M$ satisfying
\begin{itemize}
\item[(1)] $[(\omega_i^{\flat})^{-1}\circ \omega_j^{\flat}]^2 = -{\rm id}_{TM}$, $\forall i\neq j$; 
\item[(2)] the bundle map $g^{\flat} \colon TM\to T^*M$ given by $g^{\flat}=\omega^{\flat}_3(\omega^{\flat}_1)^{-1}\omega^{\flat}_2$ is  positive definite.
\end{itemize}
\end{defn}

Here, we use $\omega_i^{\flat}:TM \rightarrow T^*M$ to denote the bundle map determined by $\langle \omega_i^{\flat}(X),Y\rangle = \omega_i(X,Y)$, for any $X,Y \in T_mM$ and $m\in M$. Similarly, let $\pi$ be a Poisson bivector field. We use $\pi^{\sharp}:T^*M \rightarrow TM$ to denote the bundle map determined by $\langle \pi^{\flat}(\alpha),\beta\rangle = \pi(\alpha,\beta)$, for any $\alpha,\beta \in T^*_mM$ and $m\in M$.

Note that from any hyperkähler structure $(\omega_1, \omega_2,\omega_3)$, we obtain an equivalent description $(g,I,J,K)$ where $g$ is from condition (2), and the complex structures are given by $(\omega_i^{\flat})^{-1}\circ \omega_j^{\flat}$. If only condition (1) in the above definition is satisfied, Xu refers to $(M,\omega_1,\omega_2,\omega_3)$ as a \textbf{hypersymplectic structure}, which is also referred to as a pseudo-hyperkähler structure in \cite{Gteman2009PseudoHyperkhlerGA}.

To compare different constructions for hyperkähler structures, we establish the notion of isomorphism. The idea is that we allow for a rotation (action of $\SO(3)$) on the three symplectic forms.

\begin{defn}
Let $(M,\omega_1,\omega_2,\omega_3)$ and $(\tilde{M}, \tilde{\omega}_1,\tilde{\omega}_2,\tilde{\omega}_3)$ be two hyperkähler manifolds. A diffeomorphism ${F}\colon\tilde{M} \rightarrow M$  is called an \textbf{isomorphism of hyperkähler structures} if
\begin{equation}
F^*(\omega_1,\omega_2,\omega_3)=(\tilde{\omega}_1,\tilde{\omega}_2,\tilde{\omega}_3)O
\end{equation}
for some $O\in \SO(3).$    
\end{defn}
If we consider hypersymplectic structures on $M$ and $\tilde{M}$, then we say a diffeomorphism ${F}\colon\tilde{M} \rightarrow M$ is an \textbf{anti-isomorphism}, if equation (1) is valid for some $O^{\prime} \in \rm{O}(3)$ with $\rm{det}(O^{\prime})=-1.$ It is sometimes convenient to phrase the notion of isomorphism in terms of the Riemannian metric and complex structures. 

\begin{propn}
Equivalently, an isomorphism of hyperkähler structures $(M,g,I,J,K)$ and $(\tilde{M},\tilde{g},\tilde{I},\tilde{J},\tilde{K})$ means that: there exists an isometry ${F}\colon(\tilde{M},\tilde{g}) \rightarrow (M,g)$ such that 
\begin{equation}
(I {F}_*,J {F}_*,K {F}_*)=({F}_*\tilde{I},{F}_*\tilde{J},{F}_*\tilde{K})O
\end{equation}
for some $O\in\SO(3)$.
In particular, isomorphic hyperkähler structures correspond to the same Riemannian metric. 
\end{propn}

\begin{remark}
In the case of hypersymplectic structures, we are considering pseudo-Riemannian metrics. Then an anti-isomorphism means a diffeomorphism that preserves the metrics up to a minus sign, and intertwines the complex structures up to $O\in \SO(3).$
\end{remark}

\subsection{General theory of hyper-Lie Poisson structures}
The notion of hyper-Lie Poisson structures is a generalization of Lie-Poisson structures in the ``hyper" context, and was initially introduced by Xu in \cite{Xu}. The crux of this concept lies in constructing three compatible Poisson structures $(\pi_1,\pi_2,\pi_3)$ that share the same symplectic leaves. Consequently, on each leaf, we obtain hypersymplectic structures.

Let $\g$ be a semisimple Lie algebra and  $\langle-,-\rangle$  the Killing form on $\g.$
We shall define the Poisson structures on $\g^3=:\g\times\g\times\g$.  For any $\xi \in \g$, let $l^j_{\xi}$ be the function on $\g^3$ given by
$l^j_{\xi}(a_1,a_2,a_3) = \langle \xi,a_j \rangle$.

Let $A$ be a function $\operatorname{SO(3)}$-invariant and $S^2(\g)$-valued on $\g^3$, where $S^2(\g)$ denotes the space of second order symmetric tensors on $\g$. When contracting with $\xi,\eta\in\g$, one gets a function $A_{\xi,\eta}:=\xi\lrcorner A\lrcorner \eta$ on $\g^3$, and the $\SO(3)$-invariance means that $A_{\xi,\eta}(a_1,a_2,a_3) = A_{\xi,\eta}((a_1,a_2,a_3)O)$, for any $O\in \SO(3)$ and $(a_1,a_2,a_3) \in \mathfrak{g}^3.$ 

The following brackets define a bivector field $\pi_1$ on $\g^3$, $\forall \xi,\zeta \in \mathfrak{g},$
\begin{equation}\label{Eqt:brackets}
    \begin{aligned}
        \{l^1_{\xi},l^1_{\zeta}\}=l^1_{[\xi,\zeta]}, \hspace{1em} &\{l^1_{\xi},l^2_{\zeta}\}=\{l^2_{\xi},l^1_{\zeta}\}=l^2_{[\xi,\zeta]},\\
        \{l^2_{\xi},l^2_{\zeta}\}=-l^1_{[\xi,\zeta]}, \hspace{1em} &\{l^1_{\xi},l^3_{\zeta}\}=\{l^3_{\xi},l^1_{\zeta}\}=l^3_{[\xi,\zeta]},\\
        \{l^3_{\xi},l^3_{\zeta}\}=-l^1_{[\xi,\zeta]}, \hspace{1em} &\{l^2_{\xi},l^3_{\zeta}\}=-\{l^3_{\xi},l^2_{\zeta}\}=A_{\xi,\zeta}.
    \end{aligned}
\end{equation}
Applying pushforward by cyclic permutation $\sigma\colon(a_1,a_2,a_3)\to (a_3,a_1,a_2)$, one can obtain the other two bivector fields $\pi_2=\sigma_*\pi_1$ and $\pi_3=\sigma_*\pi_2$ on $\g^3$.

Furthermore, if we impose additional technical conditions on $A$, as outlined in \cite{Xu},  such that
\begin{itemize}
    \item $\pi_1$ is a Poisson tensor,
    \item the symplectic foliation of $\pi_1$, $\pi_2$ and $\pi_3$ coincides,
    \item any symplectic leaf $L$ is a hypersymplectic manifold,
\end{itemize} 
then we say $(\pi_1,\pi_2,\pi_3)$ is a \textbf{hyper-Lie Poisson structure} on $\g^3$. By definition, the symplectic foliations of $\pi_1, \pi_2, \pi_3$ coincide; and we call this common foliation the \textbf{hypersymplectic foliation}, since each leaf carries the induced hypersymplectic structure.

\begin{remark}
In the definition of $\pi_1$, the first bracket aligns with the Lie-Poisson structure on $\g,$, while the first three brackets correspond to the Lie-Poisson structure on $\g^{\mathbb{C}}$, viewed as a real Lie algebra. This rationale substantiates the terminology ``hyper-Lie Poisson structures". However, in contrast to the Lie-Poisson structure, $\pi_1$  exhibits nonlinearity due to the presence of $A$.
\end{remark}

For any $O\in \operatorname{O}(3)$, we define
\begin{equation}\label{Eqt:DefsigmaO}
\begin{aligned}
 \sigma_O \colon \mathfrak{g}^3 &\rightarrow \mathfrak{g}^3,\\
( a_1, a_2, a_3)
&\mapsto
( a_1, a_2, a_3)O.
\end{aligned}
\end{equation}
 Let $(\pi_1,\pi_2,\pi_3)$ be a hyper-Lie Poisson structure on $\mathfrak{g}^3.$ Xu proved the following result by tedious computation.
 
\begin{lemma}[\cite{Xu}, Theorem 3.5]\label{lem:sym-Poisson}
For any $O\in \operatorname{O}(3)$, we have that $${\sigma_O}_*(\pi_1, \pi_2, \pi_3) = (\pi_1, \pi_2, \pi_3)O.$$
\end{lemma}
 
Moreover, we have the following result.
\begin{thm}\label{thm:o3sym-gen}
Let $L$ be any symplectic leaf of $\pi_1,$ then $\sigma_O(L)$ is also a symplectic leaf, and $\sigma_O$ is a diffeomorphism between them. Moreover,
\begin{itemize}
\item if $\rm det(O)=1,$ then $\sigma_O$ restricts to an isomorphism of hypersymplectic structures,
\item if $\rm det(O)=-1,$ then $\sigma_O$ restricts to an anti-isomorphism of hypersymplectic structures.
\end{itemize}

\end{thm}

\begin{proof}
First, we show that $\sigma_O$ preserves the symplectic distribution. By the fact $\pi_1^{\sharp}(T^*\mathfrak{g}^3)=\pi_2^{\sharp}(T^*\mathfrak{g}^3)=\pi_3^{\sharp}(T^*\mathfrak{g}^3)$ and Lemma \ref{lem:sym-Poisson}, we have
\begin{equation}
    \begin{aligned}
        {\sigma_O}_*(TL)& = {\sigma_O}_*(\pi_1^{\sharp}(T^*\mathfrak{g}^3|_L)) = ({\sigma_O}_*\pi_1)^{\sharp}(T^*\mathfrak{g}^3|_{\sigma_O({L})})\\
        &= (O_{11}\pi_1^{\sharp}+O_{21}\pi_2^{\sharp}+O_{31}\pi_3^{\sharp})(T^*\mathfrak{g}^3|_{\sigma_O({L})})\\
        & = {\sigma_O}_*(\pi_1^{\sharp}(T^*\mathfrak{g}^3|_{\sigma_O(L)})) = T\sigma_O(L).
    \end{aligned}
\end{equation}
Since $L$ is a leaf, we know $\sigma_O(L)$ is a connected integral submanifold, thus contained in some leaf $L^{\prime}$. By repeating this argument with $\sigma_{O^{-1}}$ on $L^{\prime}$, we conclude that $\sigma_{O^{-1}}(\sigma_O(L))\subset \sigma_{O^{-1}}(L^{\prime}) \subset L.$ Thus $\sigma_O(L)=L^{\prime}$. 

Let $(\omega_1,\omega_2,\omega_3)$ and $(\omega^{\prime}_1,\omega_2^{\prime},\omega^{\prime}_3)$ be the induced hypersymplectic structures on $L$ and $\sigma_O(L),$ respectively. We shall see that $\sigma_O^*(\omega^{\prime}_1,\omega_2^{\prime},\omega^{\prime}_3)=(\omega_1,\omega_2,\omega_3)O^{-1}$. By Lemma \ref{lem:sym-Poisson}, 
\begin{equation}\label{eq:2.7(1)}
(({\sigma_{O^{-1}}}_*\pi_1)^{\sharp},({\sigma_{O^{-1}}}_*\pi_2)^{\sharp},({\sigma_{O^{-1}}}_*\pi_3)^{\sharp}) = (\pi_1^{\sharp}, \pi_2^{\sharp}, \pi_3^{\sharp})O^{-1}.    
\end{equation}
Since the inclusion $i:L\rightarrow \mathfrak{g}^3$ is a Poisson map, we have 
\begin{equation}\label{eq:2.7(2)}
i_*(\omega_i^{\flat})^{-1}i^* = \pi_i^{\sharp}, \hspace{1em} i_*(({\sigma_O}^*\omega^{\prime}_i)^{\flat})^{-1}i^* = ({\sigma_{O^{-1}}}_*\pi_i)^{\sharp}.   
\end{equation}

Combining Equations \eqref{eq:2.7(1)} \eqref{eq:2.7(2)}, one can obtain
$$
({(\sigma_O^*\omega^{\prime}_1)^{\flat}}^{-1},{(\sigma_O^*\omega^{\prime}_2)^{\flat}}^{-1},{(\sigma_O^*\omega^{\prime}_3)^{\flat}}^{-1})=({\omega_1^{\flat}}^{-1},{\omega_2^{\flat}}^{-1},{\omega_3^{\flat}}^{-1})O^{-1}.
$$
Therefore, we get $\sigma_O^*(\omega^{\prime}_1,\omega_2^{\prime},\omega^{\prime}_3)=(\omega_1,\omega_2,\omega_3)O^{-1}$. 
\end{proof}

Furthermore, we define another map $\iota$, given by
\begin{equation}\label{Eqt:iota}
\begin{aligned}
\iota: \mathfrak{g}^3 &\rightarrow  \mathfrak{g}^3, \\
 (a,b,c)&\mapsto (-a,-b,-c).
\end{aligned}  
\end{equation}
By Theorem \ref{thm:o3sym-gen}, we have the following:
\begin{cor}\label{prop:iota-anti}
For any leaf $L$ of $\mathfrak{g}^3,$ the restriction $\iota|_L:L\rightarrow \iota(L)$
is an anti-isomorphism of hypersymplectic structures. 
\end{cor}


\subsection{The $\su(2)$-case}
We seek an $S^2(\mathfrak{g})$-valued function $A$ for which Equation~\eqref{Eqt:brackets} defines a hyper-Lie Poisson structure. In general, the existence of such an $A$ is nontrivial. From now on, we shall only consider the special case $\g=\su(2)$.

 Let $\Phi$ be a function on $\su(2)^3$ defined by $\Phi(a_1,a_2,a_3)=\langle a_1,[a_2,a_3]\rangle$ for all $a_1,a_2,a_3\in \su(2)$. Let $\su(2)^3_{\star}$, $\su(2)^3_{+}$ and $\su(2)^3_{-}$ be the open subsets of $\su(2)^3$ consisting of all points such that $\Phi\neq 0$, $\Phi>0$ and $\Phi<0$, respectively.

In this case, following \cite{Xu}, certain $A$ is constructed on $\su(2)^3_{\star}$:
$$A(a_1,a_2,a_3)=\frac{1}{\Phi}([a_1,a_2]\otimes[a_1,a_2]+[a_2,a_3]\otimes[a_2,a_3]+[a_3,a_1]\otimes[a_3,a_1]).$$
Let $\pi_1$, $\pi_2$, $\pi_3$ be the bivector field on $\su(2)^3_{\star}$ corresponding to $A,$ as in Equations \eqref{Eqt:brackets}. 

\begin{thm}[\cite{Xu}, Theorem 4.2]
The bivector fields $(\pi_1, \pi_2, \pi_3)$ defined above are a hyper-Lie Poisson structure on $\su(2)^3_{\star}.$ As a consequence, each symplectic leaf admits a hypersymplectic structure.
\end{thm}

Next we describe the symplectic distribution and Casimir functions of the Poisson structure $\pi_1$  on $\su(2)^3_{\star}$.
By Proposition 4.5 of \cite{Xu}, the following functions
\begin{equation}\label{Eqt:casimir}
\begin{aligned}
    &f_1(a_1,a_2,a_3)=\langle a_1,a_2 \rangle,\hspace{1em} f_2(a_1,a_2,a_3)=\langle a_2,a_3 \rangle,\\  &f_3(a_1,a_2,a_3)=\langle a_3,a_1 \rangle,\hspace{1em} f_4(a_1,a_2,a_3)=\langle a_1,a_1 \rangle-\langle a_2,a_2 \rangle,\\
    & 
    f_5(a_1,a_2,a_3)=\langle a_2,a_2 \rangle-\langle a_3,a_3 \rangle,
\end{aligned}
\end{equation}
form a complete set of Casimirs for $(\su(2)^3_{\star}, \pi_1)$. The symplectic distribution $\mathcal{D} = \pi_1^{\sharp}(T^*\su(2)^3_{\star})$ is a 4-dimensional regular distribution. We have that $\mathcal{D}=\operatorname{Span}\{V_0, V_1, V_2, V_3\}$,
where $V_i\in \mathfrak{X}(\su(2)^3_{\star}),$ given by
\begin{equation}\label{Eqt:global-frame}
\begin{aligned}
V_0|_{(a_1,a_2,a_3)} &= ([a_2,a_3], [a_3,a_1], [a_1,a_2]),\\ V_1|_{(a_1,a_2,a_3)} &= (0, [a_2,a_1], [a_3,a_1]),\\
V_2|_{(a_1,a_2,a_3)} &= ([a_1,a_2], 0, [a_3,a_2]), \\
V_3|_{(a_1,a_2,a_3)} &= ([a_1,a_3], [a_2,a_3],0).
\end{aligned}    
\end{equation}
Each leaf $L$ of $(\su(2)^3_{\star};\pi_1,\pi_2,\pi_3)$ is hypersymplectic, which can be described as follows:

\begin{propn}
The complex structures $I,J,K$ on $L$ are given by the following formulae:
\begin{align}
I(V_0,V_1,V_2,V_3)&=(V_1,-V_0,V_3,-V_2),\label{Eqt:leaf-I}\\
J(V_0,V_1,V_2,V_3)&=(V_2,-V_3,-V_0,V_1),\label{Eqt:leaf-J}\\
K(V_0,V_1,V_2,V_3)&=(V_3,V_2,-V_1,-V_0).\label{Eqt:leaf-K}
\end{align}
\end{propn}
\begin{proof}
By Equations (4) and (5) of \cite{Xu}, we have 
$$I=(\omega_3^b)^{-1}\circ \omega_2^b, J=(\omega_1^b)^{-1}\circ \omega_3^b, K=(\omega_2^b)^{-1}\circ \omega_1^b.$$
By Lemma 3.9, 3.10 and Proposition 4.6 of \cite{Xu}, the Hamiltonian vector fields for the Poisson structure $\pi_1$ on $\mathfrak{su}(2)^3_{\star}$ are described by
\begin{equation}\label{Eqt: Hamiltonian_vf}
\begin{aligned}
    &X_{l^1_{a_1}} = X_{l^2_{a_2}} = X_{l^3_{a_3}}= (0, [a_2,a_1], [a_3,a_1]), \\
    & X_{l^1_{a_2}} = -X_{l^2_{a_1}}  = ([a_1,a_2], 0, [a_3,a_2]), \\
    & X_{l^1_{a_3}} = -X_{l^3_{a_1}} = ([a_1,a_3], [a_2,a_3],0),\\
    & X_{l^2_{a_3}} = - X_{l^3_{a_2}} = ([a_2,a_3], [a_3,a_1], [a_1,a_2]).
\end{aligned}
\end{equation}
Similarly, one obtains Hamiltonian vector fields for $\pi_2$ and $\pi_3$. One computes directly that
\begin{equation*}
\begin{aligned}
&I(V_0)=(\omega_3^{\flat})^{-1}\circ \omega_2^{\flat}(V_0)
=\pi_3^\sharp(0,0,a_1)=V_1,\\
&I(V_1)=(\omega_3^{\flat})^{-1}\circ \omega_2^{\flat}(-V_1)
=(\omega_3^{\flat})^{-1}(dl^2_{a_1})
=-([a_2,a_3],-[a_1,a_3],[a_1,a_2])=-V_0,\\
&I(V_2)=(\omega_3^{\flat})^{-1}\circ \omega_2^{\flat}(V_2)
=(\omega_3^{\flat})^{-1}(dl^2_{a_2})
=-([a_3,a_1],-[a_2,a_3],0)=V_3,\\
&I(V_3)=(\omega_3^{\flat})^{-1}\circ \omega_2^{\flat}(V_3)
=(\omega_3^{\flat})^{-1}(dl^2_{a_3})
=-([a_1,a_2],0,[a_3,a_2])=-V_2.
\end{aligned}
\end{equation*}

The other identities for $J$ and $K$ follow similarly.
\end{proof}

\begin{propn}
The (pseudo-)Riemannian metric $g$ on $L$ is given by 
\begin{equation}\label{Eqt:Xu-metric}
g(V_i,V_j)=-\delta_{ij}\Phi,
\end{equation}
for all $i,j = 0,1,2,3$

\end{propn}
\begin{proof}
The identity $g(V_i,V_i)=-\Phi$ was checked in Theorem 3.11 and Proposition 4.6 of \cite{Xu}. To show that $V_i, V_j$ are orthogonal whenever $i\neq j$, first note that $g = \omega_1^{\flat}\circ I^{-1}$, by Equation (5) of \cite{Xu}. The result then follows from a direct verification using Equations \eqref{Eqt:leaf-I} and \eqref{Eqt: Hamiltonian_vf}.
\end{proof}


\begin{remark}\label{rem:PD&ND}
Since the manifold $\su(2)^3_{\star}$ has two connected components $\su(2)^3_{+}$ and $\su(2)^3_{-}$, we know that $L$ lies in one of them. Therefore, the pseudo-metric is either positive definite, or negative definite.
\end{remark}

\section{Leaves of the Hyper-Lie Poisson structure on $\su(2)^3_{\star}$}
\subsection{Description of leaves}
We shall give an explicit description to any arbitrary symplectic leaf of $\pi_1$ on $\su(2)^3_{\star}$. The idea is that, since any Casimir function of a Poisson structure must remain constant on a leaf, any leaf must lie in some level set of Casimir functions. We shall make use of the Casimir functions in Equations \eqref{Eqt:casimir} and look at the corresponding level sets.

Let $f:=(f_1,f_2,f_3,f_4,f_5): \su(2)^3 \rightarrow \mathbb{R}^5$ be the direct sum of the Casimir functions in Equations \eqref{Eqt:casimir}. In this subsection, we prove the following result.
\begin{thm}\label{Thm:leaf}
For any $\eta^+\in \su(2)^3_{+}$, i.e. $\Phi(\eta^+)>0,$ the symplectic leaf of $\pi_1$ through $\eta^+$ is 
$$N_{\eta^+}:= f^{-1}(f(\eta^+))\cap \su(2)^3_{+}.$$  
Similarly, for any $\eta^-\in \su(2)^3_{-}$, i.e. $\Phi(\eta^-)<0,$ the symplectic leaf through $\eta^-$ is $$N_{\eta_-}:=f^{-1}(f(\eta^-))\cap \su(2)^3_{-}.$$
\end{thm}

First, we show that any level set of $f$ is an integral submanifold of the symplectic distribution $\mathcal{D} = \pi_1^{\sharp}(T^*\su(2)^3_{\star})$. 
\begin{propn}
For any $\eta \in \su(2)^3_{\star},$ the level set $f^{-1}(f(\eta))$ is a closed embedded submanifold of $\su(2)^3_{\star}.$ Moreover, its tangent space coincides with $\mathcal{D}$.
\end{propn}
\begin{proof}
Restrict the map $f$ to $\su(2)^3_{\star}$, then it is straightforward to verify that ${\rm Rank( Jac}(f))|_{\su(2)^3_{\star}}=5.$ Consequently, for any $\eta \in \su(2)^3_{\star},$ $f^{-1}(f(\eta))$  is a 4-dimensional closed embedded submanifold of $\su(2)^3_{\star}$.

For the second assertion, we use the fact that $Tf^{-1}(f(\eta)) = \ker(f_*)$. Moreover, one checks by direct computation that $V_i\in \ker(f_*)$, for any $i=0,1,2,3,$ as in Equations \eqref{Eqt:global-frame}. Since $\{V_0,V_1,V_2,V_3\}$ are linearly independent over any point of $\su(2)^3_{\star}$, we conclude that $$Tf^{-1}(f(\eta))= \rm Span \{V_0,V_1,V_2,V_3\} = \mathcal{D},$$ on any point of $f^{-1}(f(\eta)).$
\end{proof}

We will show the following result on connectedness.

\begin{propn}\label{Prop:connected}
For any  $\eta^+\in \su(2)^3_{+}$ and $\eta^-\in \su(2)^3_{-}$, both $N_{\eta^+}$ and $N_{\eta^-}$ are connected. Indeed, they are the connected components of the level set of $f$. 
\end{propn}

Before proving Proposition \ref{Prop:connected}, we shall see that Theorem \ref{Thm:leaf} follows from the previous results.

\begin{proof}[Proof of Theorem 3.1.]
Suppose $L$ is a leaf that intersects with $N_{\eta^+}$, then $L\subset \su(2)^3_{+}.$ Further, each Casimir function is constant on $L$, thus $L$ is contained in some level set of $f$. Then $L\subset N_{\eta^+}$.

On the other hand, we know that $N_{\eta^+}$ is a connected integral submanifold of $\mathcal{D}.$ Thus, it must also be a subset of $L$, which is a maximal connected integral submanifold. Thus, $N_{\eta^+}=L.$  Applying a similar argument, we deduce that $N_{\eta^-}$ is also a symplectic leaf.
\end{proof}

The rest of this subsection is devoted to the proof of Proposition \ref{Prop:connected}. The idea is that, we consider a class of diffeomorphisms induced by $\operatorname{O}(3)$ as in Equation \eqref{Eqt:DefsigmaO}, which preserves level sets. With the help of those diffeomorphisms, we may transform an arbitrary level set into a simpler form. Here, by `simpler form', we mean a subcollection of level sets $N_{q,r}$ for $q,r\geq 0,$ where
\begin{equation}\label{Eqt:Nqr-def}
\begin{aligned}
N_{q,r}:=\{& (a_1,a_2,a_3)\in \su(2)^3_{\star}\hspace{0.3em}| \hspace{0.3em}\langle a_1,[a_2,a_3]\rangle < 0, \hspace{0.5em} \langle a_i,a_j \rangle =0, \forall i\neq j , \\ 
&\langle a_1,a_1\rangle -\langle a_2,a_2 \rangle = q, \langle a_2,a_2\rangle -\langle a_3,a_3 \rangle = r \}.   
\end{aligned}    
\end{equation}

Note that $N_{q,r}$ can also be understood as the level set through $\eta^-_{q,r} = (e_1, \sqrt{q+1}e_2, -\sqrt{q+r+1}e_3),$ where $\{e_1,e_2,e_3\}$ is a positively-oriented orthonormal basis in $\su(2).$

\begin{propn}\label{prop: Nqr-connected}
For any $q,r \geq 0$, the level set $N_{q,r}$ is diffeomorphic to $\mathbb{R}_{>0}\times \SO(3)$, and thus connected. 
\end{propn}
\begin{proof}
By definition, any element in $N_{q,r}$ is of the form $$(\sqrt{s^2+r+q}e_1, \sqrt{s^2+r}e_2, -se_3),$$ for some $s\in \mathbb{R}_{>0}$, where $\{e_1,e_2,e_3\}$ constitutes a  positively-oriented orthonormal basis in $\su(2).$ Given a fixed isomorphism $\su(2)\cong \R^3$, we can regard   $(e_1,e_2,e_3)$  as an element in $\rm{SO}(3)$. Thus the natural map
\begin{equation*}
    \begin{aligned}
    \mathbb{R}_{>0}\times \rm{SO}(3) &\rightarrow 
        N_{q,r} \\(s, (e_1,e_2,e_3)) &\mapsto
        (\sqrt{s^2+r+q}e_1, \sqrt{s^2+r}e_2, -se_3)  
    \end{aligned}
\end{equation*}
has a smooth inverse $$(a_1,a_2,a_3)\mapsto (||a_3||, (\frac{a_1}{\sqrt{||a_3||^2+r+q}}, \frac{a_2}{\sqrt{||a_3||^2+r}},-\frac{a_3}{||a_3||})).$$
\end{proof}


\begin{remark}
Following \cite{YangZheng}, one can also prove the previous result by showing $N_{q,r} \cong T\mathbb{S}^2\setminus 0_{\mathbb{S}^2} \cong \mathbb{R}_{>0}\times \SO(3),$ where $0_{\mathbb{S}^2}$ denote the zero section of $T\mathbb{S}^2.$
\end{remark}

To show that each level set can be transformed into $N_{q,r}$ by diffeomorphism, we use the map  $\sigma_O$ as defined  in Equation \eqref{Eqt:DefsigmaO}.
The proof of the following result is a direct computation.

\begin{propn}\label{prop:so3action}
For any $O\in \operatorname{O}(3)$ and $\eta^+ \in \su(2)^3_+$, the map $\sigma_O \colon N_{\eta^+} \rightarrow N_{\sigma_O(\eta^+)}$  is  a well-defined diffeomorphism. Similarly, the result is still valid if we replace $\eta^+$ by any $\eta^- \in \su(2)^3_-$.  
\end{propn}

Moreover, we get a correspondence between the level set in $\su(2)^3_+$ and in $\su(2)^3_-$ via the map $\iota$ as defined  in Equation \eqref{Eqt:iota}.
\begin{propn}\label{prop:iota}
For any $\eta^+\in\su(2)^3_+$, the map $\iota$, given by 
\begin{equation*}
\begin{aligned}
\iota: N_{\eta^+} &\rightarrow  N_{-\eta^+} \\
 (a,b,c)&\mapsto (-a,-b,-c).
\end{aligned}  
\end{equation*}
is a diffeomorphism. 
\end{propn}

For any level set $N_{\eta^+}$ and $N_{\eta^-}$, let $\overline{N}_{\eta^+}$ and $\overline{N}_{\eta^-}$ be their closure in $\su(2)^3.$ Equivalently, $\overline{N}_{\eta^+}$ represent the intersection of the level set $f^{-1}(f(\eta^+))$ and $\{(a_1,a_2,a_3)\in\su(2)^3|\Phi(a_1,a_2,a_3) \geq 0\}.$ 

Our goal is to transform any arbitrary level set $N=N_{\eta^+}$ or $N_{\eta^-}$ into the form $N_{q,r}$, which implies connectedness. The idea is to look at elements in $\overline{N}\setminus N.$ We would like to find an $O\in \operatorname{O}(3)$ which transform such an element into an element in $\overline{N}_{q,r}\setminus N_{q,r}$, for some $q,r.$ Then the same $O$ would transform $N$ into $N_{q,r}.$ The first step is to see that $\overline{N}\setminus N$ is nonempty, and for this we recall a lemma of Xu.


\begin{lemma}[\cite{Xu}, Proposition 5.2]
For any $\zeta^+ \in \su(2)^3_+,$ let $\psi_t(\zeta^+)$ be the flow of the vector field $V_0,$ as in Equations \eqref{Eqt:global-frame}. Then 
\begin{itemize}
    \item either $\psi_t(\zeta^+)$ converges to some critical point of $\Phi$, as $t \rightarrow -\infty,$
    \item or there exists $D< 0$ such that $\Phi(\psi_D(\zeta^+))=0$ and $\Phi(\psi_t(\zeta^+))>0$ for $t \in (D,0].$ 
\end{itemize} 
The same result holds for any $\zeta^- \in \su(2)^3_-,$ when we consider $t\rightarrow \infty$ and $D>0$ in the above 2 cases, respectively.
\end{lemma}
Moreover, one verifies directly that $f(\psi_t)$ is constant with respect to $t$. Thus the flow always stays in the same level set of $f$ as it starts from. Therefore, from any point in $N=N_{\eta^+}$ or $N_{\eta^-}$, by following the flow, one can always reach a limit point in $\overline{N}\setminus N = \overline{N}\cap \Phi^{-1}(0).$ To conclude,  
\begin{cor}\label{prop:iota}
For any level set $N=N_{\eta^+}$ or $N_{\eta^-}$, we have that $\overline{N}\setminus N\neq \emptyset.$ Moreover, for any $N_{\eta^+}$, there exists a point $\mu\in\Phi^{-1}(0)$ such that  ${N}_{\eta^+}=f^{-1}(f(\mu))\cap \su(2)^3_+$. Similarly for any $N_{\eta^-}.$
\end{cor}

We may transform any point in $\Phi^{-1}(0)$ to our desired form, due to the following elementary lemma.
\begin{lemma}\label{lem: transformO}
For any $\mu\in \su(2)^3$ with $\Phi(\mu)=0,$ there exists $O\in \operatorname{SO}(3)$ such that $\mu^{\prime\prime} = \sigma_O(\mu)$ is of the form $\mu^{\prime\prime} = (\mu^{\prime\prime}_1, \mu^{\prime\prime}_2, 0)$, where $\langle \mu^{\prime\prime}_1 , \mu^{\prime\prime}_2\rangle =0$ and $\langle \mu^{\prime\prime}_1 , \mu^{\prime\prime}_1\rangle \geq\langle \mu^{\prime\prime}_2 , \mu^{\prime\prime}_2\rangle $.
\end{lemma}
\begin{proof}
Write $\mu= (\mu_1,\mu_2,\mu_3)$. Then $\Phi(\mu)=0$ implies that $\{\mu_1, \mu_2, \mu_3\}$ does not define a basis in $\su(2)$, and we assume, without loss of generality, that $\mu_3=y_1\mu_1+y_2\mu_2$. There exists a rotation matrix $O_1\in \SO(3)$ with the third column $\frac{1}{\sqrt{y_1^2+y_2^2+1}}(y_1, y_2, -1)^T,$ such that $(\mu_1, \mu_2, \mu_3)O_1 = (\mu_1^{\prime}, \mu_2^{\prime},0)$ for some $\mu_1^{\prime}, \mu_2^{\prime} \in \su(2).$ 

Assume that $(\mu_1^{\prime\prime},\mu_2^{\prime\prime},0) = (\mu_1^{\prime}, \mu_2^{\prime},0)O_2, $ where $O_2$ is of the form
\begin{equation*}
\begin{bmatrix}
cos(\theta) & sin(\theta) & 0\\
-sin(\theta) & cos(\theta) & 0\\
0 & 0 & 1
\end{bmatrix}.
\end{equation*}
Then we have
 $$\langle \mu_1^{\prime\prime},\mu_2^{\prime\prime}\rangle = \frac{sin(2\theta)}{2}(\langle \mu_1^{\prime},\mu_1^{\prime}\rangle-\langle \mu_2^{\prime},\mu_2^{\prime}\rangle)+cos(2\theta)\langle \mu_1^{\prime},\mu_2^{\prime} \rangle. $$
So there exists $\theta$ such that $\langle \mu_1^{\prime\prime},\mu_2^{\prime\prime}\rangle =0.$
We may assume that $\langle \mu_1^{\prime\prime},\mu_1^{\prime\prime}\rangle \geq \langle \mu_2^{\prime\prime},\mu_2^{\prime\prime}\rangle.$ Since if not, we can always take $O_3$ such that $(\mu_1^{\prime\prime},\mu_2^{\prime\prime},0)O_3=(\mu_2^{\prime\prime},\mu_1^{\prime\prime},0)$.
\end{proof}

\begin{propn}\label{prop:transfer-Nqr}
For any  $\eta^-\in \su(2)^3_{-}$, let $N_{\eta^-}$ be the corresponding level set. There exists $O \in \operatorname{SO}(3)$, such that $\sigma_O$ restricts to a diffeomorphism between $N_{\eta^-}$ and $N_{q,r}.$ 

Moreover, the result is still valid if we replace $\eta^-$ by $\eta^+\in \su(2)^3_+$ and consider some $O^{\prime}\in \operatorname{O}(3)$ with $\rm det(O^{\prime})=-1.$
\end{propn}
\begin{proof}
We may only prove for $N_{\eta^-}$, by applying Corollary \ref{prop:iota}. Pick any point $\mu\in \overline{N}_{\eta^-}\setminus N_{\eta^-}.$ By Lemma \ref{lem: transformO}, there exists $O\in \SO(3)$ such that $\sigma_O(\mu)$ is of the form $\mu^{\prime\prime}$ as in the Lemma. Then there exists $q,r\geq 0$ such that $\mu^{\prime\prime}\in \overline{N}_{q,r}.$  Then by Proposition \ref{prop:so3action}, the restriction of $\sigma_O$ gives a diffeomorphism between $N_{\eta^-}$ and $N_{q,r}.$
\end{proof}

Finally, Proposition \ref{Prop:connected} follows from Proposition \ref{prop: Nqr-connected} and \ref{prop:transfer-Nqr}.

\subsection{Hypersymplectic structures on leaves}
We have obtained a description of the symplectic leaves of the hyper-Lie Poisson structure on $\su(2)^3_{\star}:$ the leaf through any $\eta^+\in \su(2)^3_+,$ or $\eta^-\in \su(2)^3_-$, is given by $N_{\eta^+}$, or $N_{\eta^-}$, respectively. Following Remark \ref{rem:PD&ND}, the metric on any $N_{\eta^-}$ is positive definite, and therefore hyperkähler. On the leaves $N_{\eta^+}$, we get hypersymplectic structures with negative definite metrics.

By Theorem \ref{thm:o3sym-gen}, each $\sigma_O$ preserves the hypersymplectic structures on leaves. Further, by Proposition \ref{prop:transfer-Nqr}, each leaf can be transformed into $N_{q,r},$ under some $\sigma_O$. Therefore, the leaves are all represented by special leaves $N_{q,r}$, up to (anti)-isomorphism of hypersymplectic manifolds. We conclude that

\begin{thm}\label{thm:Nqr}
Let $L$ be any leaf of $\su(2)^3_{\star}$ of the hyper-Lie Poisson structure. 
\begin{itemize}
 \item If the metric on $L$ is positive-definite, i.e., $L\subset \su(2)^3_{-}$, then there exists $q,r\geq 0$ and an isomorphism of hyperkähler structures between $L$ and $N_{q,r}$ (see Equation \eqref{Eqt:Nqr-def}).
 \item If the metric on $L$ is negative-definite, i.e., $L\subset \su(2)^3_{+}$, then there exists $q^{\prime},r^{\prime}$ and an anti-isomorphism between $L$ and  $N_{q^{\prime},r^{\prime}}$.
\end{itemize}
\end{thm}




This theorem suggests that any hypersymplectic leaf can be represented in the special form $N_{q,r}$. Consequently, studying the hyperkähler structures on $N_{q,r}$ is sufficient for understanding all the hypersymplectic leaves. In the following section, the convenience of $N_{q,r}$ will be evident, especially in calculating curvature.

\section{Sectional curvature of leaves, $N_{q,r}$}
In this section, we further study the hyperkähler metric on leaves $N_{q,r}$ of the hyper-Lie Poisson structure on $\su(2)^3_{\star}.$ According to Theorem \ref{thm:Nqr}, those special leaves represent all the hypersymplectic leaves up to (anti-)isomorphism. Note that one cannot distinguish leaves by looking at their smooth structure, since they are all diffeomorphic, due to Proposition \ref{prop: Nqr-connected}. However, as we will see, leaves admit different behavior on sectional curvatures.


Recall that the sectional curvature on a Riemannian manifold $(M,g)$, in the direction of $u_1,u_2 \in T_pM$ is given by
$$
\kappa(u_1,u_2) := \frac{g(R(u_1,u_2)u_2,u_1)}{g(u_1,u_1)g(u_2,u_2)-(g(u_1,u_2))^2},
$$
which depend only on the two-dimensional subspace $\operatorname{Span}\{v,w\} \subset T_pM.$ Here, $R$ stands for the curvature tensor of the Levi-Civita connection of $g.$ For our purpose, we regard $\kappa$ as a function on the second Grassmannian bundle of the tangent bundle $Gr_2(TM).$ Here is the main theorem in this section.

\begin{thm}\label{thm:firstmainresult}
For any $q,r \geq 0,$ let $g_{q,r}$ be the induced metric on $N_{q,r}$, and let $\kappa_{q,r}: Gr_2(TN_{q,r})\rightarrow \mathbb{R}$ be the sectional curvature of the Levi-Civita connection of $g_{q,r}.$
\begin{itemize}
    \item On $N_{0,0}$, the sectional curvature $\kappa_{0,0}$ vanishes identically.
    \item On $N_{q,0}$ with $q>0$, $\kappa_{q,0}$ is a nonconstant function bounded by $\frac{12}{\sqrt{q}}$.
    \item On $N_{q,r}$ with $r>0$, $\kappa_{q,r}$ is an unbounded function. 
\end{itemize}
\end{thm}
Based on the different behaviors of the sectional curvatures on leaves, we define \textbf{the type of a leaf}. We say $N_{0,0}$ is of type $0$, any $N_{q,0}$ with $q>0$ is of type 1, and any $N_{q,r}$ with $r>0$ is of type 2. In this way, we divide all $N_{q,r}$ into 3 subfamilies. Since any 2 leaves with different types cannot be isometric, by Theorem \ref{thm:firstmainresult}, we get that
\begin{cor}
Any two leaves with different types cannot be isomorphic as hyperkähler manifolds.
\end{cor}
\begin{remark}
The type of any leaf $N_{q,r}$ has another intepretation. Pick any element $\mu_{q,r} \in \overline{N}_{q,r}\setminus N_{q,r}$, and write $\mu_{q,r} = (\mu_1,\mu_2,\mu_3).$ One can show that the number $\rm dim(Span\{\mu_1,\mu_2,\mu_3\})$ depend only on $q,r$, and coincides with the type of $N_{q,r}.$
\end{remark}
The rest of this section is devoted to the proof of Theorem \ref{thm:firstmainresult}.
\subsection{Riemann curvature tensor on $N_{q,r}$}
Recall from Equations \eqref{Eqt:global-frame} that each leaf $N_{q,r}$ is a 4-dimensional manifold with global frame $\{V_0,V_1,V_2,V_3\}.$ Indeed, this frame is orthogonal, due to Equation \eqref{Eqt:Xu-metric}. In this subsection, we compute the curvature tensor of the Levi-Civita connection of $g_{q,r}$ on $N_{q,r}$, on this global frame.

For any point $(a_1,a_2,a_3)\in N_{q,r}$, we define $X=\langle a_1,a_1 \rangle$, $Y=\langle a_2,a_2 \rangle$ and $Z=\langle a_3,a_3 \rangle$ so that $X,Y,Z$ are three functions on $N_{q,r}$. The following formulae describe the Lie bracket of such vector fields. 

\begin{lemma}
On each leaf $N_{q,r}$, we have 
{\begin{equation*}
\begin{aligned}
[V_0,V_1]=\frac{YZ}{\Phi}V_1,\indent &
[V_0,V_2]=\frac{XZ}{\Phi}V_2,\\
[V_0,V_3]=\frac{XY}{\Phi}V_3, \indent & [V_1,V_2]=-\frac{XY}{\Phi}V_3,\\
[V_1,V_3]=\frac{XZ}{\Phi}V_2, \indent & [V_2,V_3]=-\frac{YZ}{\Phi}V_1.
\end{aligned}
\end{equation*}}
\end{lemma}
\begin{proof}
Pick any positively-oriented orthonormal basis $\{e_1,e_2,e_3\}$ of $\su(2).$ For any point $(a_1,a_2,a_3) \in N_{q,r}\subset \su(2)^3,$
write $a_i=a_{ij}e_j$ for $i=1,2,3$. For 
any $\xi\in\su(2)$, denote by $\hat{\xi}$ the vector field $-([\xi,a_1], [\xi,a_2],[\xi,a_3])$ on $N_{q,r}$. According to the property $[V_0,\hat{\xi}]=0$ for any $\xi\in\su(2)$,  as stated in \textcolor{blue}{Equation 19 of} \cite{Xu}, we obtain:
$$[V_0,V_1]=[V_0,a_{1j}\hat{e}_j]=V_0(a_{1j})\hat{e}_j=\left | \begin{matrix}
\hat{e}_1 &\hat{e}_2   & \hat{e}_3 \\
a_{21} &a_{22} & a_{23}  \\
a_{31} & a_{32} &a_{33} \\
\end{matrix} \right |. $$
Since $\langle a_i,a_j\rangle=0$ for $i\neq j$ and $\langle a_1,[a_2,a_3]\rangle <0$, we have that $a_2 \times a_3 = -\sqrt{\frac{\langle a_2,a_2 \rangle \langle a_3,a_3 \rangle}{\langle a_1,a_1 \rangle}} a_1,$ with the cross product on $\mathbb{R}^3\cong \su(2).$ Moreover, we have that $\Phi=-\sqrt{XYZ}$ on $N_{q,r}$, and therefore
$$[V_0,V_1]=-\sqrt{\frac{\langle a_2,a_2 \rangle \langle a_3,a_3 \rangle}{\langle a_1,a_1 \rangle}}V_1=\frac{YZ}{\Phi}V_1.$$
The remaining equations can be derived using  a similar  approach.
\end{proof}

Denote by $\nabla$ the Levi-Civita connection associated with the Riemannian metric $g:=g_{q,r}$ on $N_{q,r}.$ For any vector fields $U,V$ and $W$, recall the formula 
\begin{equation}
\begin{aligned}
g( \nabla_{U}V,W) = &\frac{1}{2} \big(Ug( V,W)  + V g( U,W )- W g( U,V ) \\
&+g(  W,[U,V] ) - g(  U,[V,W] ) - g(  V,[U,W] )\big).
\end{aligned}
\end{equation}

The following lemma follows from direct computations.

\begin{lemma}\label{prop:nabla-connection}
On each leaf $N_{q,r}$, we have
\begin{align*}
\nabla_{V_0}{V_0} = \frac{-X Y - Y Z - Z X}{2\Phi}V_0, \hspace{1em} &
\nabla_{V_0}{V_1} = \frac{-X Y - Y Z - Z X}{2\Phi}V_1,\\
\nabla_{V_0}{V_2} = \frac{-X Y - Y Z - Z X}{2\Phi}V_2, \hspace{1em} &
\nabla_{V_0}{V_3} = \frac{-X Y - Y Z - Z X}{2\Phi}V_3,\\
\nabla_{V_1}{V_1} = \frac{X Y - Y Z + Z X}{2\Phi}V_0,\hspace{1em} &
\nabla_{V_1}{V_2} = \frac{X Y - Y Z + Z X}{2\Phi}V_3,\\
\nabla_{V_1}{V_3} =\frac{-XY+YZ-ZX}{2\Phi}V_2,\hspace{1em} &
\nabla_{V_2}{V_2} = \frac{X Y + Y Z - Z X}{2\Phi}V_0,\\
\nabla_{V_2}{V_3}= \frac{X Y +Y Z - Z X}{2\Phi}V_1,\hspace{1em} &
\nabla_{V_3}{V_3} = \frac{-X Y +Y Z + Z X}{2\Phi}V_0.
\end{align*}
\end{lemma}

Let $R$ be the curvature tensor of $\nabla$. Using Lemma \ref{prop:nabla-connection}, we deduce that 
\begin{align}
R(V_1,V_2)V_2 &=  \frac{X^2(-2Y^2+YZ+Z^2)+X(Y^2Z-2YZ^2)+Y^2Z^2}{2XYZ}V_1 ,\label{Eqt:R122}\\
R(V_1,V_3)V_3 &=  \frac{X^2(Y^2+YZ-2Z^2)+X(-2Y^2Z+YZ^2)+Y^2Z^2}{2XYZ}V_1,\label{Eqt:R133} \\
R(V_2,V_3)V_3 &=  \frac{X^2(Y^2-2YZ+Z^2)+X(Y^2Z+YZ^2)-2Y^2Z^2}{2XYZ}V_2.\label{Eqt:R233}
\end{align}

Denote by $R_{ijkl}$ the function
$
g(R(V_i,V_j)V_k,V_l).
$ From Equation \eqref{Eqt:Xu-metric}, it is immediate that
\begin{align}
R_{1221} &= \frac{X^2(-2Y^2+YZ+Z^2)+X(Y^2Z-2YZ^2)+Y^2Z^2}{2\Phi} ,\label{Eqt:R1221}\\
R_{1331} &=  \frac{X^2(Y^2+YZ-2Z^2)+X(-2Y^2Z+YZ^2)+Y^2Z^2}{2\Phi},\label{Eqt:R1331} \\
R_{2332} &=  \frac{X^2(Y^2-2YZ+Z^2)+X(Y^2Z+YZ^2)-2Y^2Z^2}{2\Phi}.\label{Eqt:R2332}
\end{align}
We shall study all $R_{ijkl}$ by using the above formalae, as well as the property of curvature tensor.
In fact, all nonzero $R_{ijkl}$  can be deduced from $R_{1221}, R_{1331}$ and $R_{2332}$  by utilizing the symmetry property.

For any $i,j,k,l = 0,1,2,3,$ we know that 
\begin{equation}\label{Eqt:curv-sym-basic}
\begin{aligned}
 &R_{ijkl}=R_{klij}=-R_{jikl}=-R_{ijlk}\\
=&R_{jilk}=R_{lkji}=-R_{lkij}=-R_{klji}.   
\end{aligned}
\end{equation}
In particular, if $i=j$ or $k=l$, then $R_{ijkl} = 0.$
Further, we recall the following formulae for Kähler manifolds.
\begin{lemma} \cite{kahlerbook}\label{Lem:kahler-curvature}
If $(M,g,J)$ is Kähler, then 
\begin{equation}
R(U, V,W,Q)=R(JU,JV,W,Q)=R(U,V,JW,JQ),
\end{equation}
for any vector fields $U, V,W,Q$ on $M$.
\end{lemma}

Since $N_{q,r}$ is a hyperkähler structure, we know $g$ is Kähler with respect to three complex structures. Consequently, due to Lemma \ref{Lem:kahler-curvature}, the Riemann curvature tensor exhibits the following additional symmetric properties.
\begin{propn}
On the manifold $N_{q,r}$, we have
\begin{align}
R_{0101} &=  -R_{2301}=R_{2323},\label{Eqt:0101} \\
R_{0202} &= R_{1302}=R_{1313}, \label{Eqt:0202}\\
R_{0303} &= -R_{1203}=R_{1212}\label{Eqt:0303},
\end{align}
\begin{align}
R_{0112} &=  -R_{2312}=-R_{0103}=R_{2303}, \label{Eqt:R0112}\\
R_{0102} &= -R_{2302}=R_{0113}=-R_{2313}, \label{Eqt:R0102}\\
R_{0212} &= R_{1312}=-R_{1303}=-R_{0203}.\label{Eqt:R0212}
\end{align}
\end{propn}
\begin{proof}
Using the complex structure $J$ on $N_{q,r}$, defined by  Equation \eqref{Eqt:leaf-J}, and Lemma \ref{Lem:kahler-curvature}, we have
\begin{align*}
R_{0101}&=R(V_0,V_1,V_0,V_1)=R(JV_0,JV_1,V_0,V_1)\\
&=R(V_2,-V_3,V_0,V_1)=-R_{2301},
\end{align*}
\begin{align*}
R_{0101}&=R(V_0,V_1,V_0,V_1)=R(JV_0,JV_1,JV_0,JV_1)\\
&=R(V_2,-V_3,V_2,-V_3)=R_{2323}.
\end{align*}
The remaining equations can be derived using a similar approach.
\end{proof}

\begin{propn}\label{prop:48terms}
On $N_{q,r},$ there are 48 nonzero terms of $R_{ijkl}$, for any $i,j,k,l=0,1,2,3.$ Moreover, each of the nonzero terms equals to one of $\pm R_{1221}$, $\pm R_{1331}$ and $\pm R_{2332}.$
\end{propn}
\begin{proof}
We consider 3 different cases. First, suppose $(i,j,k,l)$ consists of 2 distinct indices. In order for $R_{ijkl}$ to be nonzero, it must be of the form $R_{ijij}$ or $ R_{ijji}$ for some $i\neq j$. There are 24 terms in total. By Equation \eqref{Eqt:0101}, we get that $R_{0101} = R_{2323},$ and then by Equations \eqref{Eqt:curv-sym-basic}, we get 8 terms, which are all $\pm R_{2323}.$ Repeating this process by using Equations \eqref{Eqt:0202}\eqref{Eqt:0303}, we get the 24 terms.

Now suppose $(i,j,k,l)$ consists of 3 different indices. We prove that $R_{ijkl}=0$. Without loss of generality, we shall only prove for $R_{ijjk}$, for mutually distinct $i,j,k.$  By Equations \eqref{Eqt:R122} and \eqref{Eqt:R133}, we have
$$R_{1220}=R_{1223}= R_{1330}=R_{1332}=R_{2330}=0.$$
Then all the terms in Equations \eqref{Eqt:R0112}, \eqref{Eqt:R0102}, and \eqref{Eqt:R0212} vanish, and 
$$R_{1002}=R_{1003}=R_{2003}=R_{0112}=R_{0113}=R_{3112}=R_{3002}=0.$$
By applying Equations \eqref{Eqt:curv-sym-basic}, we see that all the 24 terms of the form $R_{ijjk}$ vanish.

Finally, the case for $(i,j,k,l)$ being mutually distinct contributes 24 nonzero terms. Each of such terms can be computed by using Equations \eqref{Eqt:0101}, \eqref{Eqt:0202}, \eqref{Eqt:0303} and \eqref{Eqt:curv-sym-basic}.
\end{proof}



\subsection{Proof of Theorem \ref{thm:firstmainresult}}
Now we are ready to apply computations from the previous subsection to prove Theorem \ref{thm:firstmainresult}. Note that by definition of $N_{q,r}$, the functions $X=\langle a_1,a_1 \rangle$, $Y=\langle a_2,a_2 \rangle$ and $Z=\langle a_3,a_3 \rangle$ satisfies the relation:
$$X-Y = q, \hspace{1em} Y-Z = r.$$

First we consider $N_{0,0}$, the only leaf of type 0. The condition reads $X=Y=Z$ and therefore by Equations \eqref{Eqt:R1221}\eqref{Eqt:R1331}\eqref{Eqt:R2332}, we have that
$$R_{1221}=R_{1331}=R_{2332}=0.$$
Then in view of Proposition \ref{prop:48terms}, we conclude

\begin{propn}
The sectional curvature $\kappa_{0,0}$ on $N_{0,0}$ vanishes identically.
\end{propn}

Next we look at the type-1 leaves, $N_{q,0}$, $q>0$. The condition $r=0$ implies that $Y=Z>0.$ Also notice that the range of $X$ is $(q,\infty).$ We first show that the sectional curvature is bounded, when evaluated on the orthogonal frame $\{V_i\}.$

\begin{lemma}\label{type1curvature}
Let $N_{q,0}$ be a leaf of type $1$. Then $$
 |\frac{R_{ijkl}}{\Phi^2}| < q^{-\frac{1}{2}},$$ for any $i,j,k,l$, and at any point $\eta \in N_{q,r}$. In particular, we have that $|\kappa_{q,r}(V_i,V_j)| = |\frac{R_{ijji}}{\Phi^2}| < q^{-\frac{1}{2}}.$
\end{lemma}
\begin{proof}
By Proposition \ref{prop:48terms}, it suffices to show that each of $R_{1221},$ $R_{1331},$ and $R_{2332}$ are bounded by $q^{-\frac{1}{2}}.$ Since $Y=Z,$ we have by Equations \eqref{Eqt:R1221}\eqref{Eqt:R1331}\eqref{Eqt:R2332} that
$$\frac{R_{1221}}{\Phi^2}=\frac{R_{1331}}{\Phi^2}=\frac{-q}{2(\sqrt{X})^3},\hspace{1em} \frac{R_{2332}}{\Phi^2}=\frac{q}{(\sqrt{X})^3}. $$ Since the range of $X$ is $(q, \infty),$ the result follows.


\end{proof}

\begin{propn}\label{prop:type1-bounded}
Let $N_{q,0}$ be a leaf of type $1$. Then the sectional curvature $\kappa$ is  bounded.
\end{propn}

\begin{proof}
For any $\eta\in N_{q,0}$, consider nonzero tangent vectors $W_1,W_2 \in T_{\eta} N_{q,0}$ that satisfy $g(W_1,W_2)=0$. Write $W_1 = \sum_{i=0}^3 b_iV_i$ and $W_2 = \sum_{i=0}^3 c_i V_i$. The sectional curvature $\kappa(W_1,W_2)$ is given by:
\begin{equation*}
    \begin{aligned}
        \kappa(W_1,W_2) &= \frac{g(R(W_1,W_2)W_2,W_1)}{g(W_1,W_1)g(W_2,W_2)-(g(W_1,W_2))^2}\\
        &=\frac{ \sum_{i,j,k,l} b_ib_l c_jc_k R_{ijkl}}{(\sum_{s=0}^3 b_s^2) (\sum_{s=0}^3 c_s^2)\Phi^2}.
    \end{aligned}
\end{equation*}
For any $i,j,k,l$ we know that $|\frac{R_{ijkl}}{\Phi^2}|$ is always bounded by $q^{-\frac{1}{2}}$. Given fixed indices $i,j,k,l$, we have the inequalities $|b_ib_l|\leq\frac{1}{2} \sum_{s=0}^3 b_s^2$  and $|c_jc_k|\leq\frac{1}{2}\sum_{s=0}^3 c_s^2$. Consequently,
$$\frac{|b_ib_lc_jc_k|}{(\sum_{s=0}^3 b_s^2) (\sum_{s=0}^3 c_s^2)}\leq \frac{1}{4}.$$ By Lemma \ref{type1curvature}, there are at most 48 combinations of $(i,j,k,l)$ for which $R_{ijkl}\neq 0$.  Therefore $|\kappa(W_1,W_2)|$ must be bounded by $12q^{-\frac{1}{2}}$.
\end{proof}

As the final step of the proof of Theorem \ref{thm:firstmainresult}, we study the leaves $N_{q,r}$ of type $2$. 

\begin{propn}\label{prop:unbounded}
For any leaves $N_{q,r}$ of type $2$, the function $\kappa_{q,r}(V_0,V_1)$ is unbounded.  
\end{propn}
\begin{proof}
For any leaf $N_{q,r}$ of type $2$, we consider two cases: either $q=0,r>0 $, or $q> 0, r> 0.$ If $q=0,r>0$, the condition $X=Y$ implies:
\begin{equation*}
\begin{aligned}
\kappa_{q,r}(V_0,V_1)=\frac{R_{0110}}{g(  V_0,V_0 )g(  V_1,V_1 )} &= \frac{Y-Z}{2(\sqrt{Z})^3} =  \frac{r}{2(\sqrt{Z})^3}.
\end{aligned}
\end{equation*}
Since the range of $Z$ is $(0,\infty),$ 
$\kappa_{q,r}(V_0,V_1)$ diverges to $\infty$ as $Z$ approaches $0$.

When $q> 0, r> 0,$ $Z$ could approach $0$, while $X$ and $Y$ converge to $q+r$ and $r$, respectively. In this case, $\kappa_{q,r}(V_0,V_1)$ is given by:
$$\dfrac{(X^2+XY-2Y^2)Z^2+(-2X^2Y+XY^2)Z+X^2Y^2}{2(\sqrt{X})^3(\sqrt{Y})^3(\sqrt{Z})^3},$$
which also diverges to infinity as $Z\rightarrow 0$. Therefore, $\kappa_{q,r}$ is unbounded on any type $2$ leaf.
\end{proof}


To further compare metrics on different $N_{q,r}$, for any $t>0$, consider
\begin{equation*}
\begin{aligned}
\rho_{t}\colon N_{q,r}&\rightarrow N_{qt,rt},\\
(a_1,a_2,a_3)&\mapsto (\sqrt{t}a_1,\sqrt{t}a_2,\sqrt{t}a_3).
\end{aligned}
\end{equation*}

By straightforward computation, one gets that
\begin{propn}
For any $t>0$, $\rho_t$ is a diffeomorphism. Moreover,  $\rho^*_t (g_{qt,rt})=\sqrt{t}\cdot g_{q,r}$, and $\rho_t$ intertwines the three complex structures, where $g_{qt,rt}$, $g_{q,r}$ are the induced metrics on $N_{qt,rt}$, $N_{q,r}$, respectively.
\end{propn}

\section{Comparison with Kronheimer's constructions}
In this section, we first recall a result of Xu, which shows that some leaves of $\su(2)^3_{\star}$ can be related to adjoint orbits of $\mathfrak{sl}(2,\mathbb{C}).$ Then we recall another way of obtaining hyperkähler structures on adjoint orbits, namely, by using moduli spaces, due to Kronheimer \cite{Kronheimer-JLMS,Kronheimer1990InstantonsAT}. Finally, we shall compare the two constructions and arrive at the conclusion that leaves of type-0 and 1 recover Kronheimer's construction, in the case of $\mathfrak{sl}(2,\mathbb{C}).$ Further, leaves of type-2 are new hyperkähler structures on the nilpotent orbit of $\mathfrak{sl}(2,\mathbb{C}).$ 

\subsection{Identification of leaves of $\su(2)^3_{\star}$ and adjoint orbits}
Let $(\pi_1,\pi_2,\pi_3)$ be the hyper-Lie Poisson structure on $\su(2)^3_{\star}.$ We know by Equations \eqref{Eqt:brackets} that the projection map is a Poisson map
\begin{equation}
\begin{aligned}
{\rm pr}_{12}: (\su(2)^3_{\star}, \pi_1) & \rightarrow  (\su(2)^{\mathbb{C}}, \pi_{\rm Lie}), \\
(a_1,a_2,a_3) & \rightarrow a_1 + i a_2.
\end{aligned}
\end{equation}
Here, $\pi_{\rm Lie}$ is the Lie-Poisson structure on $\su(2)^{\mathbb{C}}$, seen as a real Lie algebra and identified with its dual via the Killing form. Moreover, the restrictions of ${\rm pr}_{12}$ give natural identifications between leaves of $\su(2)^3_{\star}$ and adjoint orbits of $\su(2)^{\mathbb{C}}\cong \mathfrak{sl}(2,\mathbb{C}).$ The proof of the following result can be found in Section 5 of \cite{Xu}. 

\begin{propn} \label{prop:N00-nilp}
The type-0 leaf $N_{0,0}$ is diffeomorphic to the nilpotent orbit of $\mathfrak{sl}(2,\mathbb{C})$, which consists of $x\in \mathfrak{sl}(2,\mathbb{C})$ such that 
${\rm ad}(x)$ is nilpotent, under restricition of ${\rm pr}_{12}.$
\end{propn}

To obtain a similar result for the type-1 leaves, Xu proved the following result.

\begin{propn}\label{prop:Nq0_reg}
The hyperkähler structures on any type-1 leaf $N_{q,0}$ ($q>0$) can be extended to its closure $\overline{N}_{q,0}.$ Further, for any $q>0$, the projection map ${\rm pr}_{12}$ induces a diffeomorphism between $\overline{N}_{q,0}$ and some regular orbit of $\mathfrak{sl}(2,\mathbb{C}).$  
\end{propn}

\begin{remark}
By fixing a basis, one can describe the adjoint orbits. For example, consider $\{e_1,e_2,e_3\}\subset \mathfrak{sl}(2,\mathbb{C})$, given by
\begin{equation*}
e_1 = \frac{1}{2}
\begin{bmatrix}
0 & i \\
i & 0
\end{bmatrix},\hspace{1em}
e_2 = \frac{1}{2}
\begin{bmatrix}
0 & -1 \\
1 & 0
\end{bmatrix},\hspace{1em}
e_3 = \frac{1}{2}
\begin{bmatrix}
i & 0 \\
0 & -i
\end{bmatrix}.
\end{equation*}
We identify $\mathfrak{sl}(2,\mathbb{C})$ with $\mathbb{C}^3.$ Then the \textbf{nilpotent orbit} is given by 
$$
\{(x,y,z)\in \mathbb{C}^3\setminus\{(0,0,0)\}|x^2+y^2+z^2=0\}.
$$
The \textbf{regular orbit} corresponding to $\overline{N}_{q,0}$ under ${\rm pr}_{12}$ is 
$$
x^2+y^2+z^2 = q^2.
$$
\end{remark}

\begin{remark}\label{rem:singularities}
One can give an explicit description of the extension, following Section 5 of \cite{Xu}.
For any $\tau=(a,0,0) \in \overline{N}_{q,0}\setminus N_{q,0}$, we write $a = re_1,$ for some positively-oriented orthonormal basis $\{e_1,e_2,e_3\}$ of $\su(2)$. The tangent space $T_{\tau}\overline{N}$ is given by
$${\rm Span}
\{(k_1e_2+k_2e_3, k_3e_2+k_4e_3, -k_4e_2+k_3e_3), k_1,k_2,k_3,k_4\in\R \}.$$
We may denote $[k_1,k_2,k_3,k_4]$ by the tangent vector $(k_1e_2+k_2e_3, k_3e_2+k_4e_3, -k_4e_2+k_3e_3)$. Then the extended hyperkähler metric $\overline{g}_{q,0}$ is
$$g([k_1,k_2,k_3,k_4],[k'_1,k'_2,k'_3,k'_4])=\frac{1}{r}(k_1k'_1+k_2k'_2+k_3k'_3+k_4k'_4),$$
and the extended complex structures over $\tau$ is exactly the standard quaternic relations on $[k_1,k_2,k_3,k_4].$

\end{remark}

Finally we turn to the type-2 leaves. Unlike the type-1 leaves, the hyperkähler structure on any type-2 leaf is not extendable.

\begin{propn}\label{prop:Nqr-nilp}
Let $N_{q,r}$ be any leaf of type 2. Then the hyperkähler structure cannot be extended to any point in $\overline{N}_{q,r}\setminus N_{q,r}.$ Any type-2 leaf is diffeomorphic to the nilpotent orbit of $\mathfrak{sl}(2,\mathbb{C}).$   
\end{propn}

\begin{proof}
The hyperkähler structure is not extendable, since by Proposition \ref{prop:unbounded}, the sectional curvature $\kappa_{q,r}(V_0,V_1)|_{\eta}$ approaches infinity, as $\eta$ approaches $\overline{N}_{q,r}\setminus N_{q,r} = \overline{N}_{q,r} \cap \Phi^{-1}(0).$

The second assertion follows from Proposition \ref{Prop:connected} and \ref{prop:N00-nilp}.
\end{proof}



\subsection{Kronheimer's work}
We shall review Kronheimer's construction \cite{Kronheimer1990InstantonsAT,Kronheimer-JLMS} of hyperkähler structures on adjoint orbits of $\mathfrak{g}^{\mathbb{C}},$ where $\mathfrak{g}$ is any real semisimple Lie algebra. For our purpose, we will focus on the simple case $\mathfrak{g} = \su(2).$

We say that $(\tau_1,\tau_2,\tau_3)\in \su(2)^3$ is a \textbf{regular triple}, if ${\rm span}\{\tau_1,\tau_2,\tau_3\}$ is a 1-dimensional subspace of $\su(2)$. Similarly, $(\tau_1,\tau_2)\in \su(2)^2$ is a \textbf{regular pair} if the space ${\rm span}\{\tau_1,\tau_2\}$ is a 1-dimensional subspace of $\su(2)$. Consider \textbf{Nahm's equations}, i.e.,
\begin{equation}\label{Eqt:Nahm-equation}
\begin{cases}
    \dot{B_1} = -[B_2, B_3]\\
    \dot{B_2} = -[B_3, B_1]\\
    \dot{B_3} = -[B_1, B_2]\\
\end{cases}    
\end{equation}
for $(B_1, B_2, B_3)\colon(-\infty, 0] \rightarrow \su(2)^3.$  Let $(\tau_1,\tau_2,\tau_3)\in \su(2)^3$ be any regular triple, and let $M(\tau_1,\tau_2,\tau_3)$ be the corresponding solution space for Equations \eqref{Eqt:Nahm-equation} with boundary conditions 
$$\exists g\in {\rm SU}(2)~\text{such that }~\lim\limits_{t\to-\infty}B_i(t)={\rm Ad}_g\tau_i,\quad i=1,2,3.$$
By definition, if there exists $g\in \rm SU(2)$ such that $\tau_i^{\prime} = \rm Ad_g\tau_i,$ then $M(\tau_1,\tau_2,\tau_3)=M(\tau_1^{\prime},\tau_2^{\prime},\tau_3^{\prime})$.
Using gauge theory and infinite dimensional 
hyperkähler reduction, Kronheimer proved that
\begin{thm}\label{thm:Kronheimer}\cite{Kronheimer-JLMS}
Let $(\tau_1,\tau_2,\tau_3)\in\su(2)^3$ be a regular triple. Then
\begin{itemize}
\item[(1)] $M(\tau_1,\tau_2,\tau_3)$ is a smooth manifold, with a hyperkähler structure $(\tilde{g},\tilde{I},\tilde{J},\tilde{K})$;
\item[(2)] further, if $(\tau_2,\tau_3)$ is regular, then the map 
\begin{equation}
\psi\colon M(\tau_1,\tau_2,\tau_3)\to \mathcal{O}_{\tau_2+i\tau_3},\quad (B_1,B_2,B_3)\mapsto B_2(0)+iB_3(0),
\end{equation}
is a diffeomorphism. 
\end{itemize}
\end{thm}


For the nilpotent orbit of $\mathfrak{sl}(2,\mathbb{C})$, Kronheimer used the same strategy to establish the existence of a hyperkähler structure.
Denote by $M(0,0,0)$ the solution space for the Equations \eqref{Eqt:Nahm-equation} satisfying
the boundary conditions
$$\lim\limits_{s\to -\infty} B_i(s)=0,\quad i=1,2,3.$$
\begin{thm}\label{thm:Kr-nilp}\cite{Kronheimer1990InstantonsAT}
$M(0,0,0)$ is a hyperkähler manifold and is diffeomorphic to the nilpotent orbit of $\mathfrak{sl}(2,\mathbb{C})$.
\end{thm}

Next we recall the definition of the natural hyperkähler structure  $(\tilde{g},\tilde{I},\tilde{J},\tilde{K})$ on $M(\tau_1,\tau_2,\tau_3)$.   Similarly, $M(0,0,0)$ possesses an analogous structure, so we omit the details here for brevity. The following  lemma is a direct consequence of Kronheimer's quotient construction. 
\begin{lemma}\cite{Kronheimer-JLMS} \label{lem:tangentspace}
The tangent space of $M(\tau_1,\tau_2,\tau_3)$ at a point $(B_1, B_2, B_3)$ is given by solutions $(b_0,b_1,b_2,b_3)\colon(-\infty,0]\rightarrow \mathfrak{g}^4$ to the following ODE:
\begin{equation}\label{Eqt:tangent-vector}
\begin{cases}
    \dot{b}_0 +[B_1, b_1] + [B_2, b_2] + [B_3, b_3] = 0,\\
    \dot{b}_1 -[B_1, b_0] + [B_2, b_3] - [B_3, b_2] = 0,\\
    \dot{b}_2  -[B_1, b_3] - [B_2, b_0] + [B_3, b_1] = 0,\\
    \dot{b}_3 +[B_1, b_2] - [B_2, b_1] - [B_3, b_0] = 0,
\end{cases} 
\end{equation}
and satisfying the norm
$$||b||=\mathop{{\rm sup}}\limits_{t,j}(e^{\eta|t|}|b_j|)+\mathop{{\rm sup}}\limits_{t,j}(e^{\eta|t|}|\dot{b}_j|)<\infty,\quad \text{for some}~\eta>0.$$
\end{lemma}

\begin{remark}
Lemma \ref{lem:tangentspace} remains valid for $M(0,0,0)$, if we put a different norm. See \cite{Kovalev1996NahmsEA} for a proof.
\end{remark}

\begin{remark}
To better understand the above result, it is helpful to sketch Kronheimer's construction of $M(\tau_1,\tau_2,\tau_3)$ as a quotient. Let $\mathcal{A}$ be the space of solutions $(A_0,A_1,A_2,A_3): (-\infty, 0] \rightarrow \su(2)^4,$ to the following ODE:
\begin{equation}
\dot{A_i} = -[A_0,A_i]-\epsilon_{ijk}[A_j,A_k], \hspace{1em} i,j,k=1,2,3.
\end{equation}
which converges to the constant solution $(0,\rm Ad_g\tau_1, \rm Ad_g\tau_2, \rm Ad_g\tau_3)$ for some $g\in G$, while its derivative $\dot{A_i}$ converges to $0$, all with exponential decays. Such solutions can be understood as instantons with poles, and $\mathcal{A}$ is a Banach manifold. 

One can show that the gauge transformations, $\mathcal{G}$, acts on $\mathcal{A}.$ Kronheimer proved that $\mathcal{A}/\mathcal{G} \cong M(\tau_1,\tau_2,\tau_3).$ Now Lemma \ref{lem:tangentspace} states that, each tangent vector of $M(\tau_1,\tau_2,\tau_3)$ can be identified with a tangent vector of $\mathcal{A}$ which satisfies the `gauge-fixing' condition.   
\end{remark}

For any $(\tau_1,\tau_2,\tau_3)\in \su(2)^3$ be either a regular triple, or $(0,0,0)$. The hyperkähler structure $(\tilde{g},\tilde{I},\tilde{J},\tilde{K})$ on $M(\tau_1,\tau_2,\tau_3)$ is given as follows.
For any two tangent vectors $b=(b_0,b_1,b_2,b_3)$ and $c=(c_0,c_1,c_2,c_3)$, 
\begin{equation}\label{Eqt:regular-metric}
\begin{aligned}
\tilde{g}(b,c) = \sum_{j=0}^3\int_{-\infty}^0 \langle b_j(s), c_j(s)\rangle ds,
\end{aligned}    
\end{equation}
further,
\begin{equation}\label{Eqt:complex-regular}
\begin{aligned}
\tilde{I}(b_0,b_1,b_2,b_3) &= (-b_1,b_0,-b_3,b_2)\\ \tilde{J}(b_0,b_1,b_2,b_3) &= (-b_2,b_3,b_0,-b_1)\\
\tilde{K}(b_0,b_1,b_2,b_3) &= (-b_3,-b_2,b_1,b_0).
\end{aligned}    
\end{equation}

As explained in \cite{Kronheimer-JLMS}, every rotation $O\in \operatorname{SO}(3)$ induces a symmetry i.e. a hyperkähler-isomorphism. The following result can be verified by direct computation.

\begin{propn}
For any $O\in \operatorname{SO}(3),$ the natural map 
\begin{equation}
    \begin{aligned}
        \tilde{\sigma}_{O}: M(\tau_1,\tau_2,\tau_3) &\rightarrow M((\tau_1,\tau_2,\tau_3)O),\\
        (B_1,B_2,B_3)&\mapsto (B_1,B_2,B_3)O
    \end{aligned}
\end{equation}
is an isomorphism of hyperkähler structures.
\end{propn}

By applying Lemma \ref{lem: transformO}, for any regular triple $(\tau_1,\tau_2,\tau_3)$, there exists $O\in \operatorname{SO}(3)$, such that $(\tau_1,\tau_2,\tau_3)O = (0, qe_1, 0)$ for some $q>0$ and unit vector $e_1 \in \su(2).$ Thus we get an isomorphism between $M(\tau_1,\tau_2,\tau_3)$ and $M(0, qe_1, 0).$ For simplicity, we define $M_q:=M(0,qe_1,0)$ for $q\geq 0.$ 

\subsection{The main result}
Since both constructions of Kronheimer and Xu yield hyperkähler structures on adjoint orbits of $\mathfrak{sl}(2,\mathbb{C}),$ we shall study  whether they are isomorphic. Notice that, there is a natural identification between type-0 and type-1 leaves of $\su(2)^3_{\star}$ and Kronheimer's manifolds. Indeed, the identification map is the evaluation at $t=0,$ followed by a cyclic permutation.

\begin{lemma}
For any $q>0$, let $M_q= M(0,qe_1,0)$ be the solutions space as the last subsection, and let $N_{q,0}$ be leaf of type 1. Then the natural map
\begin{equation}
    \begin{aligned}
        h_q: M_q &\rightarrow \overline{N}_{q,0},\\
        (B_1,B_2,B_3) &\mapsto (B_2(0),B_3(0),B_1(0))
    \end{aligned}    
\end{equation}
is a diffeomorphism.
\end{lemma}
\begin{proof}
This follows directly from Proposition \ref{prop:Nq0_reg} and Theorem \ref{thm:Kronheimer}
\end{proof}
Similarly, by combining Proposition \ref{prop:N00-nilp} and Theorem \ref{thm:Kr-nilp}, one gets
\begin{lemma}
The map $h_0:M_0\rightarrow N_{0,0}$ which maps $(B_1,B_2,B_3)$ to $(B_2(0),B_3(0),B_1(0))$ is a diffeomorphism.
\end{lemma}
Let $(\tilde{g},\tilde{I},\tilde{J},\tilde{K})$ and $(g,I,J,K)$ be the hyperkähler structures on $M_{q,0}$ and $\overline{N}_{q,0}$, respectively (when $q=0$, we take $N_{0,0}$ instead). Our result is, the diffeomorphisms $h_q$ also preserve the hyperkähler structures. 

\begin{thm}\label{thm:main-thm-3-relation}
For any $q\geq 0,$ the diffeomorphism $h_q$ is an isomorphism of hyperkähler structures.
\end{thm}

\begin{proof}
The key point is to find a frame for $M_q$ as in the case of Equations \eqref{Eqt:global-frame}. For any $(B_1(t),B_2(t),B_3(t)) \in M_q,$ we observe that, the following paths in $\su(2)^4$ satisfies all the conditions in Lemma \ref{lem:tangentspace}, and thus defines tangent vectors of $M_q$ at $(B_1,B_2,B_3).$
\begin{equation}\label{Eqt:global-frame-2}
\begin{aligned}
\tilde{V}_0|_{(B_1,B_2,B_3)} &= (0, [B_2,B_3], [B_3,B_1], [B_1,B_2]),\\ \tilde{V}_1|_{(B_1,B_2,B_3)} &= ([B_3,B_2], 0, [B_2,B_1], [B_3,B_1]),\\
\tilde{V}_2|_{(B_1,B_2,B_3)} &= ([B_1,B_3], [B_1,B_2], 0, [B_3,B_2]), \\
\tilde{V}_3|_{(B_1,B_2,B_3)} &= ([B_2,B_1], [B_1,B_3], [B_2,B_3], 0).
\end{aligned}    
\end{equation}
Indeed, $\{\tilde{V}_i\}$ are related to $\{V_i\}$ on $N_{q,0}$ via $h_q.$ We compute that,
\begin{equation}\label{Eqt:Fstar}
F_*{\tilde{V}_0} = -\tilde{V}_0, \hspace{.5em} F_*\tilde{V}_1 = -\tilde{V}_3, \hspace{.5em} F_*\tilde{V}_2 = \tilde{V}_1, \hspace{.5em} F_*\tilde{V}_3 = \tilde{V}_2.
\end{equation}
In particular, $\{\tilde{V}_i\}$ forms a global frame over the open dense subset $h_q^{-1}(\{\Phi\neq 0\})$ of $M_q,$  

Now we compute the Riemmanian metric $\tilde{g}$ over $\{\tilde{V}_i\}$. Since
\begin{equation*}
\begin{aligned}
&\frac{d}{dt}\langle B_1, [B_2, B_3]\rangle \\
=& \langle\dot{B_1},[B_2, B_3]\rangle +\langle\dot{B_2}, [B_3, B_1]\rangle + \langle\dot{B_3}, [B_1, B_2]\rangle \\
=& -\langle[B_2, B_3],[B_2, 
 B_3]\rangle - \langle[B_3, B_1], [B_3, B_1]\rangle - \langle[B_1, B_2], [B_1, B_2]\rangle,
\end{aligned}    
\end{equation*}
it follows from Equation \eqref{Eqt:regular-metric} that
\begin{equation*}
\begin{aligned}
&\tilde{g}(\tilde{V}_i,\tilde{V}_i)\\
=& \int_{-\infty}^0 \big(\langle[B_2, B_3],[B_2, B_3]\rangle +\langle[B_3, B_1], [B_3, B_1]\rangle + \langle[B_1, B_2], [B_1, B_2]\rangle \big)dt \\
=& -\int_{-\infty}^0 \frac{d}{dt}\langle B_1, [B_2, B_3]\rangle\\
=& -\langle B_1(0), [B_2(0), B_3(0)]\rangle,
\end{aligned}    
\end{equation*}
for any $i=0,1,2,3$. Again by Equation \eqref{Eqt:regular-metric}, we have $\tilde{g}(\tilde{V}_i,\tilde{V}_j)=0$, for any $i\neq j$. This shows that $h_q^*g_{q,0} = \tilde{g}_q$ holds on an dense subset, and thus holds everywhere on $M_q.$

As for the complex structures, by Equations \eqref{Eqt:complex-regular} and \eqref{Eqt:Fstar}, one verifies that
\begin{equation}\label{Eqt:complex-relation}
F_*\circ \tilde{I} = K \circ F_*, \indent F_* \circ \tilde{J} = -I \circ F_*,\indent F_*\circ \tilde{K} = -J \circ F_*.
\end{equation}
Therefore, $h_q$ is an isomorphism of hyperkähler structures, with respect to the matrix
\begin{equation*}
O=
\begin{bmatrix}
0 & 0 & 1\\
-1 & 0 & 0\\
0 & -1 & 0
\end{bmatrix}.
\end{equation*}
\end{proof}




\begin{remark}
In the previous proof, we avoid checking the identities on $h_q^{-1}(\Phi^{-1}(0))$. Indeed, one can still achieve that by solving ODEs from Equations \eqref{Eqt:tangent-vector}. For example, let $(e_1,e_2,e_3)$ be any positively-oriented orthornormal basis in $\su(2),$ and consider the constant solution $(B_1,B_2,B_3)\equiv (0, re_1, 0),$ then the tangent vectors are described by
\begin{equation*}
\begin{aligned}
   b_0(t)&=ke^{rt}e_2 + le^{rt}e_3,\\
   b_1(t)&=-le^{rt}e_2 + ke^{rt}e_3,\\
   b_2(t)&=me^{rt}e_2 + ne^{rt}e_3,\\
   b_3(t)&=-ne^{rt}e_2 + me^{rt}e_3,
\end{aligned}
\end{equation*}
for all $k,l,m,n \in \mathbb{R}$. Then one can relate above tangent vectors with those defined in Remark \ref{rem:singularities}, via the map $h_q.$ 

\end{remark}

We conclude this section by the following theorem.

\begin{thm}
Let $N_{q,r}$ be any leaf of the hyper-Lie Poisson structure $\su(2)^3_{\star}.$ For any triple $(\tau_1,\tau_2,\tau_3)\in \su(2)^3$, which is either regular or $(0,0,0),$ and $M(\tau_1,\tau_2,\tau_3)$ be the corresponding manifold of Kronheimer.
\begin{itemize}
\item[(1)] $M(0,0,0)$ is isomorphic to the type-0 leaf $N_{0,0}$ as hyperkähler manifolds.
\item[(2)] For any regular triple $(\tau_1,\tau_2,\tau_3),$ there exists $q>0$ such that $M(\tau_1,\tau_2,\tau_3)$ is isomorphic to the extended type-1 leaf $\overline{N}_{q,0}$ as hyperkähler manifolds.
\item[(3)] Each type-2 leaf $N_{q,r}$ is not a hyperkähler submanifold of $M(\tau_1,\tau_2,\tau_3).$ 
\end{itemize} 
\end{thm}
\begin{proof}
The statements (1) and (2) follow from Proposition \ref{prop:N00-nilp}, \ref{prop:Nq0_reg}, and Theorem \ref{thm:main-thm-3-relation}. Since the sectional curvature of any type-2 leaf is unbounded, by Proposition \ref{prop:unbounded}, while on any $M(\tau_1,\tau_2,\tau_3)$, it is bounded, one can obtain the assertion of (3).
\end{proof}


\vspace{20pt}  

\noindent \textbf{Acknowledgments.} The authors would like to thank Ping Xu and Mathieu Sti\'{e}non  for fruitful discussions and useful comments.


\begin{bibdiv}
\begin{biblist}




\bib{Atiyah-Hitchin}{book}{
   author={Atiyah, M.},
   author={Hitchin, N.},
   title={The geometry and dynamics of magnetic monopoles},
   series={M. B. Porter Lectures},
   publisher={Princeton University Press, Princeton, NJ},
   date={1988},
   pages={viii+134},
   isbn={0-691-08480-7},
   review={\MR{934202}},
}



\bib{Biquard1996}{article}{
   author={Biquard, O.},
   title={Sur les \'{e}quations de Nahm et la structure de Poisson des alg\`ebres
   de Lie semi-simples complexes},
   language={French},
   journal={Math. Ann.},
   volume={304},
   date={1996},
   number={2},
   pages={253--276},
   issn={0025-5831},
   review={\MR{1371766}},
}

\bib{Biquard1997}{article}{
   author={Biquard, O.},
   author={Gauduchon, P.},
   title={Hyper-K\"{a}hler metrics on cotangent bundles of Hermitian symmetric
   spaces},
   conference={
      title={Geometry and physics},
      address={Aarhus},
      date={1995},
   },
   book={
      series={Lecture Notes in Pure and Appl. Math.},
      volume={184},
      publisher={Dekker, New York},
   },
   date={1997},
   pages={287--298},
   review={\MR{1423175}},
}



\bib{Biquard1998}{article}{
   author={Biquard, O.},
   author={Gauduchon, P.},
   title={G\'{e}om\'{e}trie hyperk\"{a}hl\'{e}rienne des espaces hermitiens sym\'{e}triques
   complexifi\'{e}s},
   language={French},
   conference={
      title={S\'{e}minaire de Th\'{e}orie Spectrale et G\'{e}om\'{e}trie, Vol. 16, Ann\'{e}e
      1997--1998},
   },
   book={
      series={S\'{e}min. Th\'{e}or. Spectr. G\'{e}om.},
      volume={16},
      publisher={Univ. Grenoble I, Saint-Martin-d'H\`eres},
   },
   date={[1998]},
   pages={127--173},
   review={\MR{1666451}},
}
	
	

\bib{Calabi1979}{article}{
   author={Calabi, E.},
   title={M\'{e}triques k\"{a}hl\'{e}riennes et fibr\'{e}s holomorphes},
   language={French},
   journal={Ann. Sci. \'{E}cole Norm. Sup. (4)},
   volume={12},
   date={1979},
   number={2},
   pages={269--294},
   issn={0012-9593},
   review={\MR{543218}},
}


\bib{Feix2001}{article}{
   author={Feix, B.},
   title={Hyperk\"{a}hler metrics on cotangent bundles},
   journal={J. Reine Angew. Math.},
   volume={532},
   date={2001},
   pages={33--46},
   issn={0075-4102},
   review={\MR{1817502}},
}

\bib{Hitchin1987}{article}{
   author={Hitchin, N. },
   title={The self-duality equations on a Riemann surface},
   journal={Proc. London Math. Soc. (3)},
   volume={55},
   date={1987},
   number={1},
   pages={59--126},
   issn={0024-6115},
   review={\MR{887284}},
}
	

\bib{Kaledin}{article}{
   author={Kaledin, D.},
   title={A canonical hyperk\"{a}hler metric on the total space of a cotangent
   bundle},
   conference={
      title={Quaternionic structures in mathematics and physics},
      address={Rome},
      date={1999},
   },
   book={
      publisher={Univ. Studi Roma ``La Sapienza'', Rome},
   },
   date={1999},
   pages={195--230},
   review={\MR{1848662}},
}
	

\bib{Kovalev1996NahmsEA}{article}{
   author={Kovalev, A. G.},
   title={Nahm's equations and complex adjoint orbits},
   journal={Quart. J. Math. Oxford Ser. (2)},
   volume={47},
   date={1996},
   number={185},
   pages={41--58},
   issn={0033-5606},
   review={\MR{1380949}},
}

\bib{Kronheimer-ALE}{article}{
   author={Kronheimer, P. B.},
   title={The construction of ALE spaces as hyper-K\"{a}hler quotients},
   journal={J. Differential Geom.},
   volume={29},
   date={1989},
   number={3},
   pages={665--683},
   issn={0022-040X},
   review={\MR{992334}},
}

\bib{Kronheimer-JLMS}{article}{
   author={Kronheimer, P. B.},
   title={A hyper-K\"{a}hlerian structure on coadjoint orbits of a semisimple
   complex group},
   journal={J. London Math. Soc. (2)},
   volume={42},
   date={1990},
   number={2},
   pages={193--208},
   issn={0024-6107},
   review={\MR{1083440}},
}

\bib{Kronheimer1990InstantonsAT}{article}{
   author={Kronheimer, P. B.},
   title={Instantons and the geometry of the nilpotent variety},
   journal={J. Differential Geom.},
   volume={32},
   date={1990},
   number={2},
   pages={473--490},
   issn={0022-040X},
}

\bib{Maciocia}{article}{
   author={Maciocia, A.},
   title={Metrics on the moduli spaces of instantons over Euclidean
   $4$-space},
   journal={Comm. Math. Phys.},
   volume={135},
   date={1991},
   number={3},
   pages={467--482},
   issn={0010-3616},
   review={\MR{1091573}},
}

\bib{Mayrand}{article}{
   author={Mayrand, M.},
   title={Hyperk\"{a}hler metrics near Lagrangian submanifolds and symplectic
   groupoids},
   journal={J. Reine Angew. Math.},
   volume={782},
   date={2022},
   pages={197--218},
   issn={0075-4102},
   review={\MR{4360006}},
}

\bib{Nakajima}{article}{
   author={Nakajima, H.},
   title={Instantons on ALE spaces, quiver varieties, and Kac-Moody
   algebras},
   journal={Duke Math. J.},
   volume={76},
   date={1994},
   number={2},
   pages={365--416},
   issn={0012-7094},
   review={\MR{1302318}},
   }

    \bib{Vergne}{article}{
   author={Vergne, M.},
   title={Instantons et correspondance de Kostant-Sekiguchi},
   language={French, with English and French summaries},
   journal={C. R. Acad. Sci. Paris S\'{e}r. I Math.},
   volume={320},
   date={1995},
   number={8},
   pages={901--906},
   issn={0764-4442},
   review={\MR{1328708}},
}
	

\bib{Xu}{article}{
   author={Xu, P.},
   title={Hyper-Lie Poisson structures},
   language={English, with English and French summaries},
   journal={Ann. Sci. \'{E}cole Norm. Sup. (4)},
   volume={30},
   date={1997},
   number={3},
   pages={279--302},
   issn={0012-9593},
   review={\MR{1443488}},
}

\bib{YangZheng}{article}{
   author={Yang, B.},
   author={Zheng, F.},
   title={Examples of complete K\"{a}hler metrics with nonnegative holomorphic
   sectional curvature},
   journal={J. Geom. Anal.},
   volume={33},
   date={2023},
   number={2},
   pages={Paper No. 47, 32},
   issn={1050-6926},
   review={\MR{4523514}},
}

\bib{Yau-CPAM}{article}{
   author={Yau, S.},
   title={On the Ricci curvature of a compact K\"{a}hler manifold and the
   complex Monge-Amp\`ere equation. I},
   journal={Comm. Pure Appl. Math.},
   volume={31},
   date={1978},
   number={3},
   pages={339--411},
   issn={0010-3640},
   review={\MR{480350}},
}

\bib{kahlerbook}{book}{place={Cambridge}, series={London Mathematical Society Student Texts}, title={Lectures on Kähler Geometry}, publisher={Cambridge University Press}, author={Moroianu, A.}, year={2007}, collection={London Mathematical Society Student Texts}}

\bib{Gteman2009PseudoHyperkhlerGA}{article}{
  title={Pseudo-Hyperk{\"a}hler Geometry and Generalized K{\"a}hler Geometry},
  author={M. G{\"o}teman and U. Lindstr{\"o}m},
  journal={Letters in Mathematical Physics},
  year={2009},
  volume={95},
  pages={211-222}
}
	


\end{biblist}
\end{bibdiv}
\end{document}